\documentclass[reqno,11pt,twoside]{amsart}

\usepackage{xypic,hyperref,todonotes,comment}
\usepackage{graphics,amssymb,mathtools}
\usepackage{latexsym}
\usepackage{mathrsfs}
\xyoption{curve}
\usepackage{hyperref}
\hypersetup{ 
colorlinks,
citecolor=black,
filecolor=blue,
linkcolor=black,
urlcolor=blue
}

\newtheorem{lemma}{Lemma}[section]
\newtheorem{proposition}[lemma]{Proposition}
\newtheorem{theorem}[lemma]{Theorem}
\newtheorem{corollary}[lemma]{Corollary}

\newtheorem*{theoremI}{Theorem I}
\newtheorem*{theoremII}{Theorem II}
\newtheorem*{proposition'}{Proposition}
\newtheorem*{question'}{Question}

\theoremstyle{definition}

\newtheorem{definition}[lemma]{Definition}
\newtheorem{remark}[lemma]{Remark}
\newtheorem{question}[lemma]{Question}

\newcommand{\C}{\mathbb{C}}

\newcommand{\Aut}{\mathrm{Aut}}
\newcommand{\ox}{\otimes}

\newcommand{\g}{\mathsf{g}}
\newcommand{\h}{\mathsf{h}}
\newcommand{\p}{\mathsf{p}}
\newcommand{\E}{\mathbf{E}}
\newcommand{\F}{\mathbf{F}}
\newcommand{\G}{\mathcal{G}}
\newcommand{\N}{\mathcal{N}}
\newcommand{\rep}{\mathrm{rep}}
\renewcommand{\O}{\mathscr{O}}
\renewcommand{\L}{\mathcal{L}}
\renewcommand{\sl}{\mathsf{sl}}
\renewcommand{\S}{\mathtt{S}}

\DeclareMathSymbol{\Mu}{\mathalpha}{operators}{"4D}
\DeclareMathSymbol{\Zeta}{\mathalpha}{operators}{"5A}

\title[]{Small quantum groups associated to Belavin-Drinfeld triples}
\date{\today}

\author{Cris Negron}
\thanks{This work was supported by NSF Postdoctoral Research Fellowship DMS-1503147}
\email{negronc@mit.edu}
\address{Department of Mathematics, Massachusetts Institute of Technology, Cambridge, MA 02139}

\begin{document}
\maketitle

\begin{abstract}
For a simple Lie algebra $\g$ of type $A$, $D$, $E$ we show that any Belavin-Drinfeld triple on the Dynkin diagram of $\g$ produces a collection of Drinfeld twists for Lusztig's small quantum group $u_q(\g)$.  These twists give rise to new finite-dimensional factorizable Hopf algebras, i.e. new small quantum groups.  For any Hopf algebra constructed in this manner, we identify the group of grouplike elements, identify the Drinfeld element, and describe the irreducible representations of the dual in terms of the representation theory of the parabolic subalgebra(s) in $\g$ associated to the given Belavin-Drinfeld triple.  We also produce Drinfeld twists of $u_q(\g)$ which express a known algebraic group action on its category of representations, and pose a subsequent question regarding the classification of all twists.
\end{abstract}

\section*{Introduction}

Let $\g$ be a simple Lie algebra over $\C$ of type $A$, $D$, $E$, and let $\Gamma$ be its Dynkin diagram.  A Belavin-Drinfeld triple on $\Gamma$ is a choice of two subgraphs $\Gamma_1$ and $\Gamma_2$ and an isomorphism $T:\Gamma_1\to \Gamma_2$ satisfying a certain nilpotence condition.  In~\cite[Ch. 6]{BD} Belavin and Drinfeld showed that such a triple gives rise to solutions to the classical Yang-Baxter equation in $\g\ox\g$, and in~\cite{ESS} Etingof, Schedler, and Schiffmann showed that any Belavin-Drinfeld triple gives rise to (Drinfeld) twists of the Drinfeld-Jimbo quantum group $U_{\hbar}(\g)$.  Such a twist $J$ of $U_{\hbar}(\g)$ produces a new quantum group $U_{\hbar}(\g)^J$ and new $R$-matrix, i.e. solution to the Yang-Baxter equation (see Section~\ref{sect:twinfo}).  These new solutions to the Yang-Baxter equation quantize the classical solutions of Belavin and Drinfeld, in the sense described in~\cite{EK,ESS}.  Furthermore, one can show that any twist of the Drinfeld-Jimbo quantum group, over $\C[\! [\hbar]\! ]$, arises as one of the quantizations of~\cite{ESS}, up to gauge equivalence.
\par

Here we follow the methods of~\cite{ESS,ABRR} to produce twists of Lusztig's small quantum group $u_q(\g)$ from Belavin-Drinfeld triples.  We also produce explicit twisted automorphisms of $u_q(\g)$ which arise out of an algebraic group action on its category of representations.  The action we consider here first appeared in the work of Arkhipov and Gaitsgory~\cite{arkhipovgaitsgory03}, but can also be derived from De Concini and Kac's earlier quantum coadjoint action~\cite{deconcinikac89}, as is explained in Section~\ref{sect:G} below.  Using the Belavin-Drinfeld twists, and those twists associated to the algebraic group action, we propose a question regarding the classification of all twists of the small quantum group.

\subsection*{Belavin-Drinfeld triples and twists of $u_q(\g)$}

Recall that the small quantum group is a finite dimensional quasitriangular Hopf algebra produced from the Cartan data for $\g$ and a primitive $l$th root of unity $q$.\footnote{We will need $l$ to be coprime to a small number of integers throughout this work.}  In addition to a triple $(\Gamma_1,\Gamma_2,T)$ for $\g$ we need one more piece of data $\S$.  The element $\S$ is a choice of solution to a certain equation involving $T$, which we describe below (see Section~\ref{sect:JBD}).  Given any Belavin-Drinfeld tiple $(\Gamma_1,\Gamma_2,T)$ we will have $\max\{1,l|\Gamma-\Gamma_1|(|\Gamma-\Gamma_1|-1)/2\}$ such solutions $\S$.  We show

\begin{theoremI}[\ref{thm:twists}]
Any Belavin-Drinfeld triple $(\Gamma_1,\Gamma_2,T)$ for $\g$ and solution $\S$ produces a twist $J=J_{T,\S}$ for the small quantum group $u_q(\g)$, and an associated Hopf algebra $u_q(\g)^J$.
\end{theoremI}

This twist $J_{T,\S}$ is given explicitly by the formula
\[
J_{T,\S}=(T_+\ox 1)(R)\dots (T^n_+\ox 1)(R)\S^{-1}\Omega_{L^\perp}^{-1/2}(T^n\ox 1)(\Omega)^{-1}\dots(T\ox 1)(\Omega)^{-1}
\]
where $R$ is the $R$-matrix for $u_q(\g)$, $\Omega$ is an element representing the Killing form, and $T_+$ is an extension of $T$ to an endomorphism of the positive quantum Borel in $u_q(\g)$.  The above theorem is a non-dynamical analog of~\cite[Sect. 5.2]{EN01}, and a discrete version of~\cite[Cor. 6.1]{ESS}.
\par

Recall that for any twist $J$ of a Hopf algebra $H$ we will have a canonical equivalence between the associated tensor categories of finite dimensional representations $\mathrm{rep}(H)\overset{\sim}\to \mathrm{rep}(H^J)$.  In addition to  studying the relationship between Belavin-Drinfeld triples and solutions to the Yang-Baxter equation (i.e. $R$-matrices) for finite dimensional Hopf algebras, we want to study variances of Hopf structures under tensor-equivalence.  With this purpose in mind we give an in depth study of the Hopf algebras $u_q(\g)^J$ arising from our twists.
\par

For the remainder of the introduction fix $J=J_{T,\S}$ the twist associated to some Belavin-Drinfeld data $(\Gamma_1,\Gamma_2,T)$ and $\S$.  Using the new $R$-matrix for $u_q(\g)^J$, in conjunction with the frameworks of~\cite{radford93}, we identify the grouplike elements of $u_q(\g)^J$, show that the Drinfeld element for $u_q(\g)^J$ is equal to that of the untwisted algebra $u_q(\g)$, verify invariance of the traces of the powers of the antipode under the twists $J=J_{T,\S}$, and classify irreducible representations of the dual.  (See Corollaries~\ref{cor:grouplikes},~\ref{cor:u},~\ref{cor:S}, and Theorem~\ref{thm:dual} below.)  Our analyses of the Drinfeld element and antipode give positive answers to some general questions of~\cite{negronng} and~\cite{shimizu15} in the particular case of Belavin-Drinfeld twists of the small quantum group.  We describe our result on irreducibles more explicitly below.
\par

In the statement of the following theorem we let $\p^{ss}$ be the semisimple Lie algebra associated to the Dynkin diagram $\Gamma_1$ appearing in the Belavin-Drinfeld triple $(\Gamma_1,\Gamma_2, T)$.

\begin{theoremII}[\ref{thm:dual}]
There is an abelian group $\L$ of order $l(|\Gamma-\Gamma_1|)$ and bijection
\[
\mathrm{Irrep}\left(\C[\L]\ox u_q(\p^{ss})\right)\overset{\cong}\to \mathrm{Irrep}\left((u_q(\g)^J)^\ast\right)
\]
induced by a surjective algebra map $(u_q(\g)^J)^\ast\to \C[\L]\ox u_q(\p^{ss})$.
\end{theoremII}

By comparison, for the untwisted algebra $u_q(\g)$ we have that $\mathrm{Irrep}\left(u_q(\g)^\ast\right)=(\mathbb{Z}/l\mathbb{Z})^{|\Gamma|}$, and the representation theory of the dual is rather banal from the perspective or irreducibles and the fusion rule.  After twisting $u_q(\sl_{n+1})$, for example, we can have a copy of the rather rich category $\mathrm{rep}\left(u_q(\sl_n)\right)$ in the category of representations for the dual $(u_q(\sl_{n+1})^J)^\ast$.  This will specifically be the case for (what we call) maximal triples on $A_n$.  One should compare this result to~\cite[Thm. 5.4.1]{EN01}.

\subsection*{The Arkhipov-Gaitsgory action and twisted automorphisms}

Take ${\Theta}$ the connected, simply connected, semisimple algebraic group with Lie algebra $\g$.  In~\cite{arkhipovgaitsgory03} Arkhipov and Gaitsgory show that the category $\rep(u_q(\g))$ is tensor equivalent to a de-equivariantization of the category of corepresentations of the quantum function algebra $\O_q({\Theta})$.  The de-equivariantization is a certain (non-full) monoidal subcategory in $\mathrm{Coh}({\Theta})$ which inherits a natural action of ${\Theta}$ by left translation (see~\cite{agp14,EGNO}).  From the aforementioned equivalence we then get an action of ${\Theta}$ on $\rep(u_q(\g))$.
\par

According to general principles, any autoequivalence of $\rep(u_q(\g))$ should be expressible as a twisted automorphism $(\phi,J)$, i.e. a pair of a twist $J$ and a Hopf isomorphism $\phi:u_q(\g)\to u_q(\g)^J$.  Hence, the action of ${\Theta}$ should generate twists of $u_q(\g)$.
\par

In Section~\ref{sect:G} we show that any simple root $\alpha$ of $\g$, or its negation $-\alpha$, has an associated $1$-parameter family of twisted automorphisms $(\exp_{\pm\alpha}^\lambda,J^\lambda_{\pm\alpha})$, which then give a $1$-parameter subgroup $\omega_{\pm\alpha}$ in the group of autoequivalences of $\rep(u_q(\g))$.  We identify these $1$-parameter subgroups $\omega_{\pm\alpha}$ with the action of Arkhipov and Gaitsgory.

\begin{proposition'}[\ref{prop:G}]
For $\gamma_{\pm\alpha}:\mathbb{C}\to {\Theta}$ the $1$-parameter subgroup given by exponentiating the root space $\g_{\pm\alpha}$, we have a diagram
\[
\xymatrixrowsep{3mm}
\xymatrix{
 & {\Theta}\ar[drr]^{\mathrm{AG\ actn}} &\\
\mathbb{C}\ar[ur]^{\gamma_{\pm\alpha}}\ar[rrr]|(.4){\omega_{\pm\alpha}} & & & \mathrm{Aut}(\rep(u_q(\g))),
}
\]
where $\omega_{\pm\alpha}$ is the $1$-parameter subgroup specified by the twisted automorphisms $(\exp_{\pm\alpha}^\lambda,J^\lambda_{\pm\alpha})$.
\end{proposition'}

This result allows us to produce an explicit action of ${\Theta}$ on the collection $\mathtt{TW}(u_q(\g))$ of gauge equivalence classes of twists.  We let $\mathtt{BD}(u_q(\g))\subset \mathtt{TW}(u_q(\g))$ denote the subcollection of Belavin-Drinfeld twists $\{J_{T,\S}\}_{T,\S}$.  We pose the following question, which is also raised in~\cite{DEN}.

\begin{question'}[\ref{quest:G}]
Do the Belavin-Drinfeld twists and the $1$-parameter subgroups $\{(\exp^\lambda_{\pm \alpha},J^\lambda_{\pm\alpha})\}_{\lambda,\alpha}$ generate all twists of $u_q(\g)$?  Equivalently, is the inclusion $\mathtt{BD}(u_q(\g))\cdot {\Theta}\to \mathtt{TW}(u_q(\g))$ an equality?
\end{question'}

As was stated above, for the Drinfeld-Jimbo algebra $U_\hbar(\g)$ one can show that the Belavin-Drinfeld twists are the only twists, up to gauge equivalence.  So the appearance of ${\Theta}$ here already marks a deviation from the generic setting.

\subsection*{Organization}

Sections~\ref{sect:bg} and~\ref{sect:twinfo} are dedicated to background.  In Section~\ref{sect:JBD} we introduce and prove Theorem I.  In Sections~\ref{sect:radford} and~\ref{sect:radford2} we analyze relations between Radford's left and right subalgebras $R^J_{(l)}$ and $R^J_{(r)}$ in $u_q(\g)^J$ and the quantum parabolics associated to $\Gamma_1$ and $\Gamma_2$.  We prove an explicit description of the $R^J_{(\ast)}$ in Section~\ref{sect:proof}, which leads to the proof of Theorem II in Section~\ref{sect:reptheory}.  In Section~\ref{sect:u} we discuss the Drinfeld element and antipode of such a twist $u_q(\g)^J$.  Section~\ref{sect:G} is dedicated to the action of the algebraic group ${\Theta}$ on $\rep(u_q(\g))$.

\subsection*{Acknowledgements}
Thanks to Pavel Etingof for many helpful conversations and for providing a simplification of the proof of Proposition~\ref{prop:preu}.

\section{The small quantum group, Belavin-Drinfeld triples, and associated subgroups in the Cartan}
\label{sect:bg}

We introduce the small quantum group $u_q(\g)$, then give some information on the Cartan subgroup $G=G(u_q(\g))$ and Belavin-Drinfeld tiples.

\subsection{The small quantum group}

Take $\g$ a simple and simply laced Lie algebra, i.e. a Lie algebra of type $A$, $D$, $E$.  Let $\Phi$ be a root system for $\g$ (in the dual of some Cartan), $\Gamma$ be a choice of simple roots, and $l$ be an odd integer coprime to the determinant of the Cartan matrix for $\g$.  Let $(?,?)$ be the scaling of the Killing form so that each $(\alpha,\beta)$ is the Cartan integer for simple roots $\alpha,\beta$.  Take $q$ a primitive $l$th root of unity.
\par

The small quantum group $u_q(\g)$ is the Hopf algebra
\[
u_q(\g)=\C\langle K_\alpha,E_\alpha,F_\alpha:\alpha\in \Gamma\rangle/(\mathrm{Rels}),
\]
where $\mathrm{Rels}$ is the set of relations
\[
[K_{\alpha},K_{\beta}]=0,\ \  K_{\alpha}E_\beta=q^{(\alpha,\beta)}E_\beta K_\alpha,\ \   K_\alpha F_\beta=q^{-(\alpha,\beta)} F_\beta K_\alpha,
\]
\[
{[E_\alpha,F_\beta]}=\delta_{\alpha,\beta}\frac{K_\alpha-K_\alpha^{-1}}{q-q^{-1}},
\]
\[
{[E_\alpha,E_\beta]}=[F_\alpha,F_\beta]=0\ \ \mathrm{when}\ \ (\alpha,\beta)=0,
\]
\[
E^2_\alpha E_\beta-(q+q^{-1})E_\alpha E_\beta E_\alpha+E_\beta E_\alpha^2,\ \ \mathrm{when}\ \ (\alpha,\beta)=-1,
\]
\[
F^2_\alpha F_\beta-(q+q^{-1})F_\alpha F_\beta F_\alpha+F_\beta F_\alpha^2\ \   \mathrm{when}\ \ (\alpha,\beta)=-1.
\]
\begin{equation}\label{eq:173}
K_\alpha^l=1,\ \  E_\mu^l=F_\mu^l=0\ \ \forall\ \mu\in \Phi^+.
\end{equation}
We will explain the (currently opaque) relations~\eqref{eq:173} more clearly below.  The coproduct is given by
\[
\Delta(K_{\alpha})=K_\alpha\ox K_\alpha,\ \ \Delta(E_\alpha)=E_\alpha\ox 1+ K_\alpha\ox E_\alpha,\ \ \Delta(F_\alpha)=F_\alpha\ox K_{\alpha}^{-1}+1\ox F_\alpha
\]
and the antipode is given by
\[
S(K_\alpha)=K_{\alpha}^{-1},\ \ S(E_\alpha)=-K^{-1}_\alpha E_\alpha,\ \ S(F_\alpha)=-F_\alpha K_\alpha.
\]

We let $G$ denote the group of grouplikes in $u_q(\g)$, $u^+$ and $u^-$ denote the subalgebras generated by the $E_\alpha$ and $F_\alpha$ respectively, and $u_+$ and $u_-$ denote the positive and negative quantum Borels in $u_q(\g)$.  Note that $G$ is generated by the $K_\alpha$ and that under the adjoint action of $G$ on $u^\pm$ we will have $u_{\pm}=u^{\pm}\rtimes \C[G]$.  Note also that $u_q(\g)$ and the $u_{\pm}$ are graded by the root lattice $\mathbb{Z}\cdot \Gamma$, where the generators $E_{\alpha}$, $F_{\alpha}$, and $K_{\alpha}$ have degrees $\alpha$, $-\alpha$, and $0$ respectively.
\par

We would like to employ Lusztig's standard basis for $u_q(\g)$, which we review here.  Recall that for a reduced expression $w=\sigma_{\alpha_1}\dots\sigma_{\alpha_t}$ of the longest word $w$ in the Weyl group, in terms of the simple reflections, we have $\mathrm{length}(w)=|\Phi^+|$ and 
\begin{equation}\label{eq:order}
\Phi^+=\{\sigma_{\alpha_1}\dots\sigma_{\alpha_{i-1}}(\alpha_i):1\leq i\leq \mathrm{length}(w)\}.
\end{equation}
(See e.g.~\cite{zhelobenko87}.)  For each simple root $\alpha$ there is an automorphism $B_{\alpha}$ of $u_q$ so that the $B_{\alpha}$ together give an action of the braid group $B(\Gamma)$ on $u_q$~\cite{lusztig90}.\footnote{Our $B_\alpha$ are the $T_\alpha$ from~\cite{lusztig90}.}  Now for each $\mu\in \Phi^+$ we take
\[
E_{\mu}=B_{\alpha_1}\dots B_{\alpha_{i-1}}(E_{\alpha_i}),\ \ \mathrm{and}\ \ F_{\mu}=B_{\alpha_1}\dots B_{\alpha_{i-1}}(F_{\alpha_i}),
\]
where $\mu=\sigma_{\alpha_1}\dots\sigma_{\alpha_{i-1}}(\alpha_i)$.  The $E_\mu$ and $F_{\mu}$ defined here are the elements appearing in the above relations~\eqref{eq:173}.

\begin{theorem}[\cite{lusztig90}]\label{thm:basis}
For each $\mu\in \Phi^+$, the element $E_{\mu}$ (resp. $F_{\mu}$) is homogeneous of degree $\mu$ (resp. $-\mu$) with respect to the root lattice grading on $u_q(\g)$.  Furthermore, the collection of elements
\[
\{\prod_{\mu\in\Phi^+}E_\mu^{n_\mu}:0\leq n_\mu\leq l-1\},\ \ \{\prod_{\nu\in\Phi^+}F_\nu^{m_\nu}:0\leq m_\nu\leq l-1\},
\]
give $\C$-bases for $u^+$ and $u^-$ respectively, and
\[
\{(\prod_{\nu\in\Phi^+}F_\nu^{m_\nu})(\prod_{\mu\in\Phi^+}E_\mu^{n_\mu}):0\leq n_\mu,m_\nu\leq l-1\}
\]
gives a $\C[G]$-basis for $u_q(\g)$. 
\end{theorem}

Homogeneity of the $E_\mu$ is equivalent to the statement that $E_\mu$ is a linear combination of permutations of the monomial $E_{\alpha_1}\dots E_{\alpha_k}$, where $\mu=\alpha_1+\dots+\alpha_k$ with the $\alpha_i\in \Gamma$.  The analogous statement holds for the $F_\nu$ as well.  We note that the homogeneity is not covered in~\cite{lusztig90}, but can easily be seen from the fact that each braid group operator $B_\alpha$ is such that $\deg(B_\alpha(a))=\sigma_\alpha(\deg(a))$, for homogeneous $a\in u_q$.  From the identifications $u_\pm=u^\pm\rtimes \C[G]$ the $\C$-bases for $u^{\pm}$ produce $\C[G]$-bases for the quantum Borels.
\par

We recall finally that $u_q(\g)$ is quasi-triangular.  The $R$-matrix is
\[
R=\prod_{\mu\in\Phi^+}\left(\sum_{n=0}^{l-1}q^{-n(n+1)/2}\frac{(1-q^2)^n}{[n]_q!} E^n_{\mu}\ox F^n_{\mu}\right)\Omega,
\]
where $[n]_q!$ is the standard $q$-factorial, $\Omega\in \mathbb{C}[G]\ox \mathbb{C}[G]$ is such that $(\mu\ox \nu)(\Omega)=q^{(\mu,\nu)}$ for each $\mu,\nu\in G^\vee$, and the product is ordered with respect to the ordering on $\Phi^+$ given by~\eqref{eq:order} (see~\cite{turaev}).

\subsection{Belavin-Drinfeld triples and subgroups of $G$}
\label{sect:bdt}

We recall some information from~\cite{EN01}.  Let $\h$ be the Cartan subalgebra in $\g$, for $\g$, $\Phi$, $\Gamma$, and $l$ as above.  Let $\h_{\mathbb{Z}}^\ast=\mathbb{Z}\cdot \Gamma$ be the root lattice in $\h^\ast$.  Take now
\[
\G=\h^*_{\mathbb{Z}}/l\h^*_{\mathbb{Z}}\ \ \mathrm{and}\ \ G=\G^{\vee}.
\]
Recall that, since $\g$ is simply laced, there is a unique scaling of the Killing form on $\h^*$ which produces an integer valued form on $\h^\ast_\mathbb{Z}$ with $(\alpha,\alpha)=2$ for each $\alpha\in \Gamma$.  By our assumption on $l$ this scaling of the Killing form induces a non-degenerate symmetric bilinear form on the quotient $(?,?):\G\times \G\to \mathbb{Z}/l\mathbb{Z}$.  This gives an identification $\G\to G$, $\alpha\mapsto (\alpha,?)$, and via this identification we get an induced form on $G$.  We take $K_\gamma=(\gamma,?)$ for each $\gamma\in \G$, so that $(K_\gamma,K_\mu)=(\gamma,\mu)$.
\par

As our notation suggests, we identify $G$ with the collection of grouplike elements in $u_q(\g)$.  Since $G$ will be identified with a set of units in an algebra, we adopt a multiplicative notation $K_\alpha K_\beta=K_{\alpha+\beta}$.
\par

The following structure was introduced by Belavin and Drinfeld in~\cite{BD}.

\begin{definition}
A Belavin-Drinfeld triple (BD triple) on $\Gamma$ is a choice of two subsets $\Gamma_1,\Gamma_2\subset \Gamma$ and inner product preserving bijection $T:\Gamma_1\to \Gamma_2$ which satisfies the following nilpotence condition: for each $\alpha\in \Gamma_1$ there exists $n\geq 1$ with $T^n(\alpha)\in \Gamma-\Gamma_1$.
\end{definition}

We often take $T\alpha=T(\alpha)$.  Having fixed some BD triple $(\Gamma_1,\Gamma_2,T)$ we can define a number of subgroups in $\G$ and $G$.  Via the sequence $\Gamma\to \h^\ast_{\mathbb{Z}}\to \G$ we identify $\Gamma$ with a basis for $\G$.  We take
\[
\L=\left(\mathbb{Z}/l\mathbb{Z}\cdot \{\alpha-T \alpha:\alpha\in\Gamma_1\}\right)^{\perp},
\]
where the perp is calculated with respect to the form, and
\[
L=\left(\mathbb{Z}/l\mathbb{Z}\cdot \{K_\alpha K^{-1}_{T\alpha}:\alpha\in\Gamma_1\}\right)^{\perp}.
\]
Under the identification $\G\to G$ the subgroups $\L$ and $L$ are identified.  We take also $\G_i=\mathbb{Z}/l\mathbb{Z}\cdot \Gamma_i$ and the $G_i=\mathbb{Z}/l\mathbb{Z}\cdot \{K_\alpha:\alpha\in\Gamma_i\}$.
\par

We assume that $l$ is such that restrictions of the form to $\mathbb{Z}/l\mathbb{Z}\cdot\{\alpha-T\alpha:\alpha\in\Gamma_1\}$ and $\G_i$ are non-degenerate.  To find such an $l$ one simply considers the determinants of the (integer) matrices $[(\alpha-T\alpha,\beta-T\beta)]_{\alpha,\beta\in\Gamma_1}$ and $[(\alpha,\beta)]_{\alpha,\beta\in \Gamma_1}$ and chooses $l$ coprime to these determinants.  This will give $\L^\perp=\mathbb{Z}/l\mathbb{Z}\cdot\{\alpha-T\alpha:\alpha\in\Gamma_1\}$ and split $\G$ and $G$ as $\G=\L\times \L^\perp=\G_i\times\G_i^\perp$, $G=L\times L^\perp=G_i\times G_i^\perp$.  We also assume $l$ is such that the restriction of the form to $\G_i\times\L^\perp$ is non-degenerate, which we can do by~\cite[Lem 3.1]{ESS}, and which can be checked by considering the determinant of the corresponding matrix.  The following lemma was covered in~\cite[Sect. 5.2]{EN01} (see also~\cite[Cor 3.2]{ESS}).

\begin{lemma}\label{lem:GT}
Under the above assumptions on $l$, there are splittings $\G=\G_1\times \L$ and $\G=\G_2\times \L$, and a unique extension of $T:\Gamma_1\to \Gamma_2$ to a group automorphism $T:\G\to \G$ with $T|\L=id_{\L}$.  This automorphism preserves the form on $\G$.
\end{lemma}

We will denote this extension of $T$ to an automorphism on $\G$ simply by $T$.  By a further abuse of notation we let $T$ also denote the induced automorphism on the dual.  That is, $T:G\to G$ is the map $K_\alpha\mapsto K_{T(\alpha)}$.  Preservation of the form means specifically $(T\mu,T\nu)=(\mu,\nu)$ for each $\mu,\nu\in \G$ and $(T\ox T)(\Omega)=\Omega$.
\par

Throughout this work we make copious use of the dualities
\[
G\leftrightsquigarrow \G,\ \ L \leftrightsquigarrow \L,\ \ G_i \leftrightsquigarrow \G_i,\ \ L^\perp \leftrightsquigarrow \L^\perp,\ \ G_i^\perp \leftrightsquigarrow \G_i^\perp.
\]
By this we mean both that the duality functor $(?)^\vee$ sends the group on the left to the group on the right, and vice-versa, and that for any $K_\mu$ in the group on the left the function $(K_\mu,?)$ will be an element in the corresponding group on the right, and vice-versa.

\section{Twists and $R$-matrices}
\label{sect:twinfo}

A (Drinfeld) twist of a Hopf algebra $H$ is a unit $J\in H\ox H$ which satisfies the dual cocycle condition
\[
(\Delta\ox 1)(J)(J\ox 1)=(1\ox \Delta)(J)(1\ox J)
\]
and $(\epsilon\ox 1)(J)=(1\ox\epsilon)(J)=1$.  From such a $J$ we can define a new Hopf algebra $H^J$ which is equal to $H$ as an algebra and has the new comultiplication
\[
\Delta^J(h)=J^{-1}\Delta(h)J.
\]
The antipode on $H^J$ is given by
\[
S_J(h)=Q_J^{-1}S(h)Q_J,
\]
where $Q_J=m((S\ox 1)(J))$ and $Q_J^{-1}=m((1\ox S)(J^{-1}))$ and $m$ is multiplication.  (See e.g.~\cite{EGNO,radford11}.)
\par

Recall that a quasitriangular Hopf algebra is a Hopf algebra $H$ with a unit $R\in H\ox H$ satisfying $R\Delta(h)R^{-1}=\Delta^{op}(h)$ for all $h\in H$, as well as the relations $(\Delta\ox 1)(R)=R_{13}R_{23}$ and $(1\ox\Delta)(R)=R_{13}R_{12}$.  We have the additional relations
\[
(\epsilon\ox 1)(R)=(1\ox\epsilon)(R)=1,\ \ (S\ox 1)(R)=(1\ox S^{-1})(R)=R^{-1}
\]
and $R_{12}R_{13}R_{23}=R_{23}R_{13}R_{12}$~\cite{drinfeld89}.  When $H$ is quasitriangular with $R$-matrix $R$, the twist $H^J$ will naturally be quasitriangular with new $R$-matrix
\[
R^J=J_{21}^{-1}RJ.
\]

\subsection{Bicharacters and twists on group rings}

Let $\Lambda$ be a finite abelian group.  We call an element $B\in \C[\Lambda]\ox\C[\Lambda]=\C[(\Lambda^\vee\times\Lambda^\vee)^\vee]$ a (symmetric, antisymmetric, etc.) bicharcter if its restriction $B:\Lambda^\vee\times\Lambda^\vee \to \C^\times$ is a (symmetric, antisymmetric, etc.) bicharacter.  An easy direct check verifies

\begin{lemma}
Any bicharacter $B\in \C[\Lambda]\ox\C[\Lambda]$ is a twist for $\C[\Lambda]$.
\end{lemma}

Indeed, up to so-called gauge equivalence, every twist of an abelian group ring is given by an antisymmetric bicharacter (see e.g.~\cite{guillotkasselmasuoka12}).

\section{Twists from Belavin-Drinfeld triples}
\label{sect:JBD}

For the remainder of this study we fix $\g$ a simply laced simple Lie algebra with root system $\Phi$ and a choice of simple roots $\Gamma$.  We take $l$ as in Section~\ref{sect:bdt} and $u_q=u_q(\g)$.
\par

Let $(\Gamma_1,\Gamma_2,T)$ be a BD triple.  Following~\cite{ESS,EN01}, we extend the group maps $T^{\pm 1}:G\to G$ constructed in Lemma~\ref{lem:GT} to Hopf endomorphisms of the quantum Borels $T_{\pm}:u_{\pm}\to u_{\pm}$ defined by
\[
T_{+}(E_\alpha)=\left\{\begin{array}{ll}
E_{T\alpha} & \mathrm{when\ }\alpha\in\Gamma_1\\
0 & \mathrm{when\ }\alpha\in\Gamma-\Gamma_1,\end{array}\right.
\
T_{-}(F_\beta)=\left\{\begin{array}{ll}
E_{T^{-1}\beta} & \mathrm{when\ }\beta\in\Gamma_2\\
0 & \mathrm{when\ }\beta\in\Gamma-\Gamma_2.\end{array}\right.
\]
There will be a unique minimal positive integer $n$ such that $T_{\pm}^n|I_{\pm}=0$, where $I_{\pm}$ is the ideal in $u_{\pm}$ generated by all the $E_{\alpha}$, or $F_\alpha$.  We call this integer the {\it nilpotence degree} of $T_{\pm}$.
\par

We will be interested in antisymmetric bicharacters $\S$ in $\C[G]\ox \C[G]$ solving the following equation:
\begin{equation}\label{eq:seq}
\S^2(\alpha-T\alpha,?)=\Omega(\alpha+T\alpha,?)\ \ \forall\ \alpha\in\Gamma_1.\tag{EQ--S}
\end{equation}
We verify below that such solutions always exist, and that there are exactly $|L\wedge_{\mathbb{Z}}L|$ of them, which is expected from~\cite{BD,ESS}.
\par

This section is dedicated to a proof of the following theorem.

\begin{theorem}\label{thm:twists}
Consider any Belavin-Drinfeld triple $(\Gamma_1,\Gamma_2,T)$ and solution $\S$ to~\eqref{eq:seq}.  The element
\[
J_{T,\S}=(T_+\ox 1)(R)\dots (T^n_+\ox 1)(R)\S^{-1}\Omega_{L^\perp}^{-1/2}(T^n\ox 1)(\Omega)^{-1}\dots(T\ox 1)(\Omega)^{-1}
\]
is a twist for the small quantum group $u_q(\g)$, where $n$ is the nilpotence degree of $T_+$.
\end{theorem}

This result is a non-dynamical version of~\cite[Sect. 5.2]{EN01}, and a discrete version of~\cite[Thm. 6.1]{ESS}.  To clarify our previous point, we have

\begin{lemma}\label{lem:solex}
Antisymmetric bicharacter solutions $\S$ to equation~\eqref{eq:seq} always exist, and there are exactly $|L\wedge_{\mathbb{Z}}L|$ such solutions.
\end{lemma}

\begin{proof}
We decompose $G$ as $L\oplus L^\perp$ to get $G\wedge_{\mathbb{Z}}G=(L^\perp\wedge_{\mathbb{Z}}G)\oplus(L\wedge_{\mathbb{Z}}L)$.  Since
\[
L^\perp=(\L^\perp)^\vee=(\mathbb{Z}/l\mathbb{Z}\cdot \{\alpha-T\alpha:\alpha\in\Gamma_1\})^\vee
\]
we see that the equation~\eqref{eq:seq} specifies uniquely an element $\S_0$ in $L^\perp\wedge_{\mathbb{Z}}G$, which we extend to a bicharacter on $\G$ which vanishes on $\L\times \L$.  Whence we have found a solution to~\eqref{eq:seq}.  We can add arbitrary elements of $L\wedge_{\mathbb{Z}}L$ to arrive at the set of all solutions $\S_0+L\wedge_{\mathbb{Z}}L$.
\end{proof}

One should note that when $\mathrm{rank}(L)=|\Gamma|-1$, the solution $\S$ will be unique.  Using our nilpotence assumption on $T$ one sees that, up to an automorphism of the Dynkin diagram, this occurs only in type $A$ for the triple
\[
\Gamma=A_n,\ \Gamma=\{\mathrm{first\ }n-1\mathrm{\ roots}\},\ \Gamma_2=\{\mathrm{last\ }n-1\mathrm{\ roots}\},\ T(\alpha_i)=\alpha_{i+1}.
\]
We will call this the maximal triple on $A_n$.
\par

We will need the following basic property of the $R$-matrix.

\begin{lemma}
The $R$-matrix for $u_q(\g)$ satisfies $(T_+\ox 1)(R)=(1\ox T_-)(R)$.
\end{lemma}

\begin{proof}
Any element $W\in u_+\ox u_-$ is uniquely specified by the corresponding function $W:u_+^\ast\ox u_-^\ast\to \C$ and subsequent map $t_W:u_+^\ast\to u_-$, $f\mapsto (f\ox 1)(W)$.  By~\cite[Prop. 2]{radford93} and the fact that $T_\pm$ is a Hopf map, we see that when $W=(T_+\ox 1)(R)$ or $(1\ox T_-)(R)$ the $t_W:u_+^\ast\to u_-$ are algebra morphisms.  Since $u_+$ is cordically graded, $u_+^\ast$ is generated in degrees $0$ and $-1$ as an algebra, with $(u_+^\ast)_0=\C[G]^\ast$ and $(u_+^\ast)_{-1}=(\sum_{\alpha}\C[G] E_\alpha)^\ast$, and we see that the $t_W$ are determined by the restrictions $t_W|(u_+^\ast)_0$ and $t_W|(u_+^\ast)_{-1}$.  These restrictions are in turn determined by the homogeneous pieces
\[
(T_+\ox 1)(R)_0=(T\ox 1)(\Omega),\ \ (1\ox T_-)(R)_0=(1\ox T^{-1})(\Omega)
\]
and
\[
(T_+\ox 1)(R)_1=(q^{-1}-q)(\sum_{\alpha\in \Gamma_1}E_{T\alpha}\ox F_\alpha)(T\ox 1)(\Omega),
\]
\[
(1\ox T_-)(R)_1=(q^{-1}-q)(\sum_{\beta\in \Gamma_2}E_{\beta}\ox F_{T^{-1}\beta})(1\ox T^{-1})(\Omega),
\]
where we grade $u_+\ox u_-$ by the degree on $u_+$.  By $T$-invariance of the form $\Omega$, it follows that $(T_+\ox 1)(R)_0=(1\ox T_-)(R)_0$ and $(T_+\ox 1)(R)_1=(1\ox T_-)(R)_1$.  Whence we have the proposed equality.
\end{proof}

\subsection{General outline}

In order to prove Theorem~\ref{thm:twists} we will basically repeat the arguments of~\cite{ESS,EN01}, and so only sketch some of the unoriginal details.
\par

Fix a data $(\Gamma_1,\Gamma_2,T)$.  Following the suggestions of~\cite[Remark 6.1]{ESS}, and the general approach of~\cite{ABRR}, we will show that $J$ is a twist by showing that both
\[
(\Delta\ox 1)(J)(J\ox 1)\ \ \mathrm{and} \ \ (1\ox\Delta)(J)(1\ox J)
\]
solve a certain ``mixed ABRR" equation.  Solutions to this equation with a specified ``initial condition" are shown to be unique, so that we will have
\[
(\Delta\ox 1)(J)(J\ox 1)=(1\ox\Delta)(J)(1\ox J).
\]
\par

\begin{remark} Our presentation is slightly more complicated than that of~\cite{ESS}.  This is a result of our choice to avoid the use of dynamical twists.
\end{remark}

\subsection{Discrete ABRR in $2$-components}

For a given solution $\S$ let $\Zeta$ denote the restriction of $\S$ to $\L\times \L^\perp$, $\Sigma$ denote the restriction to $\L^\perp\times \L$, and take
\[
Q=\Zeta\Sigma=\S|\left((\L\times \L^\perp)+(\L^{\perp}\times \L)\right).
\]
We view $\Zeta$, $\Sigma$, and $Q$ as bicharacters on $\G$ by letting them vanish on all other factors of $\G\times\G$.  We also let $\Omega_L$ denote the restriction $\Omega|(\L\times\L)$.

\begin{definition}
We define $A_L^2$ and $A_R^2$ the be the linear endomorphisms of $u_+\ox u_-$ defined by
\[
A_L^2(\xi)=(T_+\ox 1)(R\xi Q)Q^{-1}\Omega_L^{-1},\ \ A_L^2(\xi)=(1\ox T_- )(R\xi Q)Q^{-1}\Omega_L^{-1}.
\]
The left and right $2$-component ABRR equations are the equations $A_L^2(X)=X$ and $A_R^2(X)=X$ respectively.
\end{definition}

We note that $Q$ can be replaced with $\Sigma$ and $\Zeta$ in the expressions for $A^2_L$ and $A^2_R$ respectively.  These alternate expressions are preferable for some calculations.
\par

Since $R$ decomposes as a sum $R=\Omega+R_+$, where $R_+$ is in the nilpotent ideal $I_+\ox I_-$, we get a corresponding decomposition of $A_L^2$ as
\[
A_L^2(\xi)=(T_+\ox 1)(\Omega \xi Q)Q^{-1}\Omega^{-1}_L+(T_+\ox 1)(R_+\xi Q)Q^{-1}\Omega^{-1}_L.
\]
From this one finds that we can solve the left $2$-component ABRR equation provided we can solve to the equation $(T\ox 1)(\Omega X_0Q)Q^{-1}\Omega^{-1}_L=X_0$ in $\C[G]\ox \C[G]$.  The analogous stamentement holds for the right ABRR equation.  Whence we have the following discrete analog of~\cite[Cor. 4.1]{ESS}.

\begin{lemma}[{cf.~\cite{ESS}}]
Any solution $B\in \C[G]\ox \C[G]$ to the equation
\begin{equation}\label{eq:347}
(T\ox 1)(\Omega X_0Q)Q^{-1}\Omega_L^{-1}=X_0\ \ (\text{resp.}\ (1\ox T^{-1})(\Omega X_0Q)Q^{-1}\Omega_L^{-1}=X_0)
\end{equation}
extends uniquely to a solution $J\in B+I_+\ox I_-$ to the ABRR equation $A_L^2(X)=X$ (resp. $A_R^2(X)=X$).
\end{lemma}

In the proof of the following lemma we use the fact that for any element $K_\mu\in \C[G]$, and $\nu\in\G$, we have
\[
(T^{\pm 1}(K_\mu))(\nu)=K_\mu(T^{\mp1}\nu).
\]
This follows from the easy sequence
\[
(T^{\pm 1}(K_\mu))(\nu)=K_{T^{\pm 1}\mu}(\nu)=(T^{\pm1}\mu,\nu)=(\mu,T^{\mp1}\nu)=K_\mu(T^{\mp1}\nu).
\]

\begin{lemma}\label{lem:355}
There are unique solutions $J_L,J_R\in \S^{-1}\Omega^{-1/2}_{L^{\perp}}+I_+\ox I_-$ to the left and right $2$-component ABRR equations, respectively.
\end{lemma}

\begin{proof}
We are claiming first that $\S^{-1}\Omega^{-1/2}_L$ solves the degree $0$ ABRR equations from the previous lemma.  Reorganizing, and applying $T^{-1}\ox 1$, we see that $\S^{-1}\Omega^{-1/2}_L$ solves ABRR on the left, say, if and only if the equation
\[
\Omega\Omega_{L^{\perp}}^{-1/2}(T^{-1}\ox 1)(\Omega_L^{-1}\Omega^{1/2}_{L^{\perp}})=\S Q^{-1}(T^{-1}\ox 1)(\S^{-1}Q)
\]
is satisfied.  Using the fact that $T|L=id_L$ and $\Omega=\Omega_L\Omega_{L^{\perp}}$ we reduce to
\[
\Omega_{L^{\perp}}^{1/2}(T^{-1}\ox 1)(\Omega_{L^{\perp}}^{1/2})=\S Q^{-1}(T^{-1}\ox 1)(\S^{-1}Q).
\]
Applying to arbitrary elements $\mu,\nu\in\G$ gives the equivalent equation
\begin{equation}\label{eq:369}
\Omega^{1/2}_{L^\perp}(\mu+T\mu,\nu)=\S Q^{-1}(\mu-T\mu,\nu).
\end{equation}
By writing $\mu$ as a sum of elements in $\L$ and $\L^\perp$ we see that the above equation holds if and only if it holds when $\mu\in \L$, or $\mu\in \L^\perp$.  When $\mu\in \L$ both sides of equation~\eqref{eq:369} vanish since $T|\L=id_{\L}$.  Suppose now $\mu\in \L^\perp$.  When $\nu\in \L$ both sides of the equation vanish by the definition of $Q$, and when $\nu\in \L^\perp$ the equation reduces to
\[
\Omega^{1/2}(\mu+T\mu,\nu)=\S(\mu-T\mu,\nu),
\]
which holds by equation~\eqref{eq:seq}.  The check on the right is similar.
\end{proof}

\begin{lemma}\label{lem:401}
The elements $J_L$ and $J_R$ from Lemma~\ref{lem:355} are equal.  Rather, there is a unique simultaneous solution $J$ to both the left and right $2$-component ABRR equations in $\S^{-1}\Omega_{L^\perp}^{-1/2}+I_+\ox I_-$. 
\end{lemma}

\begin{proof}
One shows that the operators $A^2_L$ and $A^2_R$ commute, then proceeds as in~\cite[Cor. 4.1]{ESS}.
\end{proof}

We find now

\begin{lemma}[{cf.~\cite[Prop. 3.3]{ESS}}]
Our proposed twist $J_{T,\S}$ solves both the left and right $2$-component ABRR equations.
\end{lemma}

\begin{proof}
Let $J$ denote the solution from Lemma~\ref{lem:401}.  We have $J=B+J_+$, where $B=\S^{-1}\Omega_{L^\perp}^{-1/2}$ and $J_+\in I_+\ox I_-$.  From the appearance of $T_+$ in $A^2_L$, and the fact that $J=A_L^2(J)$, we have
\[
J=(A^2_L)^n(J)=(A^2_L)^n(B)+(A^2_L)^n(J_+)=(A^2_L)^n(B),
\]
where $n$ is the nilpotence degree of $T_+$.  One establishes the equality
\[
(A^2_L)^k(B)=(T_+\ox 1)(R)\dots (T^k_+\ox 1)(R)B(T^k\ox 1)(\Omega)^{-1}\dots(T\ox 1)(\Omega)^{-1}
\]
by induction on $k$, using the fact that $(T\ox 1)(\Omega BQ)Q^{-1}\Omega_L^{-1}=B$.  This gives $(A^2_L)^n(B)=J_{T,\S}$.
\end{proof}

\subsection{The $3$-component and mixed ABRR equations}

For any element $\xi\in u_q\ox u_q$ take $\xi_{12,3}=(\Delta\ox 1)(\xi)$ and $\xi_{1,23}=(1\ox\Delta)(\xi)$.  So the dual cocycle equation for a twist now appears as $J_{12,3}J_{12}=J_{1,23}J_{23}$, where $J_{12}$ and $J_{23}$ are $J\ox 1$ and $1\ox J$ respectively.

\begin{definition}
Take  $A_L^3$ and $A_R^3$ to be the linear endomorphisms of $u_+\ox u_q\ox u_{-}$ defined by
\[
\begin{array}{l}
A_L^3(\eta)=(T_+\ox 1\ox 1)(R_{13}R_{12}\eta Q_{12}Q_{13})Q_{13}^{-1}Q_{12}^{-1}(\Omega_L)_{13}^{-1}(\Omega_L)_{12}^{-1}\\\\
A_R^3(\eta)=(1\ox 1\ox T_-)(R_{13}R_{23}\eta Q_{13}Q_{23})Q_{23}^{-1}Q_{13}^{-1}(\Omega_L)_{23}^{-1}(\Omega_L)_{13}^{-1}.
\end{array}
\]
The left and right $3$-component ABRR equations are the equations $A^3_L(X)=X$ and $A^3_R(X)=X$.
\end{definition}

Let us fix $J=J_{T,\S}$.

\begin{lemma}
The elements $J_{1,23}J_{23}$ and $J_{12,3}J_{12}$ solve the left and right $3$-component ABRR equations respectively.
\end{lemma}

\begin{proof}
Take $T_1=(T_+\ox 1\ox 1)$.  We claim first that $A_L^3(J_{1,23}J_{23})=A_L^3(J_{1,23})J_{23}$.  Note that we may replace $Q$ with $\Sigma=\S|\L^\perp\times \L$ in the equation for $A_L^3$, and that for any bicharacter $B$ we have $B_{12}B_{13}=B_{1,23}$.  Also recall that for any cocommutative element $h\in H$ we will have $R\Delta(h)=\Delta(h)R$, and that since $\Sigma\in \C[L^\perp\times L]$ we will have $\Sigma_{1,23}=(1\ox T_+^k\ox 1)(\Sigma_{1,23})$ for any nonnegative integer $k$.  Using these facts together, along with the particular form of $J=J_{T,\S}$, one see that
\[
\begin{array}{rl}
T_1(J_{23}\Sigma_{1,23})\Sigma_{1,23}^{-1}(\Omega_L)^{-1}_{1,23}&=J_{23}(T_1(\Sigma))_{1,23}\Sigma_{1,23}^{-1}(\Omega_L)^{-1}_{1,23}\\
&=(T_1(\Sigma))_{1,23}\Sigma_{1,23}^{-1}(\Omega_L)^{-1}_{1,23}J_{23},
\end{array}
\]
which implies $A_L^3(J_{1,23}J_{23})=A_L^3(J_{1,23})J_{23}$.  We now note that
\[
J_{1,23}=(1\ox\Delta)(J)=(1\ox\Delta)(A_L^2(J))=A^3_L(J_{1,23})
\]
to see $A_L^3(J_{1,23}J_{23})=J_{1,23}J_{23}$.  The equality $A_R^3(J_{12,3}J_{12})=J_{12,3}J_{12}$ is proved similarly.
\end{proof}

As was the case for the $2$-component equations, one finds that solutions to the equations $A_L^3(X)=X$ and $A_R^3(X)=X$ are uniquely determined by their components in $\C[G]\ox u_q\ox u_-$ and $u_-\ox u_q\ox \C[G]$ respectively.  (See also~\cite[Lem. 4.3]{ESS}.)  We also consider the mixed ABRR equation
\[
A_L^3A_R^3(X)=X.
\]
Solutions to this equation are uniquely determined by their component in $\C[G]\ox u_q\ox \C[G]$, which we denote $X_{0,0}$.  Note that
\begin{equation}\label{eq:492}
(J_{12,3}J_{12})_{0,0}=(J_{1,23}J_{23})_{0,0}=\S^{-1}_{12}\S_{13}^{-1}\S_{23}^{-1}(\Omega_{L^\perp}^{-1/2})_{12}(\Omega_{L^\perp}^{-1/2})_{13}(\Omega_{L^\perp}^{-1/2})_{23}.
\end{equation}
So we would like to establish

\begin{proposition}\label{prop:mixedABRR}
Both $J_{12,3}J_{12}$ and $J_{1,23}J_{23}$ solve the mixed ABRR equation $A_L^3A_R^3(X)=X$.
\end{proposition}

From this proposition one easily finds the proof of Theorem~\ref{thm:twists}.  We only prove the proposition for $J_{1,23}J_{23}$, the situation for $J_{12,3}J_{12}$ being completely analogous.  Let us first give some technical lemmas.  Recall $\Zeta=\S|\L\times \L^\perp$.

\begin{lemma}\label{lem:516}
The element $\Zeta$ solves the following equations:
\begin{enumerate}
\item[(i)] $(1\ox T^{-1})(\S^{-1}\Omega_{L^\perp}^{1/2})=\S^{-1}\Omega^{-1/2}_{L^\perp}(1\ox T^{-1})(\Zeta^{-1})\Zeta$.
\item[(ii)] $\left[(\Omega_L)_{13}(1\ox 1\ox T^{-1})(\Zeta^{-1}_{13})\Zeta_{13},(1\ox 1\ox T_-^k)(R_{23})\right]=0$ for all $k\geq 1$.
\end{enumerate}
\end{lemma}

\begin{proof}
Equation (i) is equivalent to the equation
\[
\Omega^{1/2}_{L^\perp}(1\ox T^{-1})(\Omega_{L^\perp}^{1/2})=\S^{-1}(1\ox T^{-1})(\S)(1\ox T^{-1})(\Zeta^{-1})\Zeta,
\]
which is seen to hold by~\eqref{eq:seq}, just as in the proof of Lemma~\ref{lem:355}.  For (ii) first note that for any $\mu\in \G$ and $\nu\in \G_1$ we have
\[
\begin{array}{l}
{\left(\Omega_L(1\ox T^{-1})(\Zeta^{-1})\Zeta\right)}(\mu,\nu)\\
\hspace{1cm}=\Omega(\bar{\mu},\nu)\Zeta(\mu,\nu-T\nu)\\
\hspace{1cm}=\Omega(\bar{\mu},\nu)\S(\bar{\mu},\nu-T\nu)\\
\hspace{1cm}=\Omega(\bar{\mu},\nu)\Omega^{-1/2}(\bar{\mu},\nu+T\nu)=\Omega(\bar{\mu},\nu)\Omega^{-1}(\bar{\mu},\nu)=1,
\end{array}
\]
where $\bar{\mu}$ is the component of $\mu$ in $\L$ under the decomposition $\G=\L\times\L^\perp$.  So we see that the bicharacter in question vanishes on $\G\times \G_1$, and hence
\[
\Omega_L(1\ox T^{-1})(\Zeta^{-1})\Zeta\in \C[G]\ox\C[G_1^\perp].
\]
It follows that all elements in $\C\ox u_q\ox \left(\C\langle G,F_\beta:\beta\in \Gamma_1\rangle\right)$ centralize $(\Omega_L)_{13}(1\ox 1\ox T^{-1})(\Zeta^{-1}_{13})\Zeta_{13}$.  Since $(1\ox 1\ox T_-^k)(R_{23})$ is in this subspace we have (ii).
\end{proof}

We can now give the

\begin{proof}[Proof of Proposition~\ref{prop:mixedABRR}]
As noted above, we only prove that $J_{1,23}J_{23}$ solves the mixed ABRR equation.  Since this element already satisfies $A_L^3(X)=X$ it suffices to show that it also solves $A_R^3(X)=X$.  As in~\cite[Lem. 4.2]{ESS}, one checks that $A_L^3$ and $A_R^3$ commute so that $A_R^3(J_{1,23}J_{23})$ solves the left ABRR equation.  
By uniqueness of solutions we find that $A_R^3(J_{1,23}J_{23})=J_{1,23}J_{23}$ if and only if these elements have the same component in $\C[G]\ox u_q\ox u_-$.  Let $A_R^3(J_{1,23}J_{23})_0$ and $(J_{1,23}J_{23})_0$ denote these components.  Take $(T_-)_3=(1\ox 1\ox T_-)$ and $B=\S^{-1}\Omega^{-1/2}_{L^\perp}$.
\par

Since $J$ is in the subalgebra $\C\langle G\times G,E_{\alpha}\ox F_\beta:\alpha,\beta\in\Gamma\rangle$ we see that $(J_{1,23}J_{23})_0=B_{1,23}J_{23}$, and we need to establish
\[
A_R^3(J_{1,23}J_{23})_0=B_{1,23}J_{23}.
\]
We have
\[
\begin{array}{rl}
A_R^3(J_{1,23}J_{23})_0&=(T_-)_3(\Omega_{13}R_{23}B_{1,23}J_{23} \Zeta_{12,3})\Zeta_{12,3}^{-1}(\Omega_L)_{12,3}^{-1}\\
&=(T_-)_3(\Omega_{13}B_{1,23}R_{23}J_{23} \Zeta_{12,3})\Zeta_{12,3}^{-1}(\Omega_L)_{12,3}^{-1}.
\end{array}
\]
Use the equality $B\Omega_{L^\perp}=\S^{-1}\Omega_{L^\perp}^{1/2}$ and Lemma~\ref{lem:516} (i) to get
\[
\begin{array}{l}
A_R^3(J_{1,23}J_{23})_0\\
=(T_-)_3((\Omega_L)_{13}B_{12}B_{13}(\Omega_{L^\perp})_{13}R_{23}J_{23} \Zeta_{12,3})\Zeta_{12,3}^{-1}(\Omega_L)_{12,3}^{-1}\\
=(\Omega_L)_{13}B_{12}(T_-)_3(B_{13}(\Omega_{L^\perp})_{13}R_{23}J_{23} \Zeta_{12,3})\Zeta_{12,3}^{-1}(\Omega_L)_{12,3}^{-1}\\
=(\Omega_L)_{13}B_{12}B_{13}(T_-)_3(\Zeta^{-1}_{13})\Zeta_{13}(T_-)_3(R_{23}J_{23} \Zeta_{12,3})\Zeta_{12,3}^{-1}(\Omega_L)_{12,3}^{-1}.
\end{array}
\]
Since $J$ solves the $2$-component ABRR equations this final expression reduces to
\[
A_R^3(J_{1,23}J_{23})_0=(\Omega_L)_{13}B_{1,23}(T_-)_3(\Zeta^{-1}_{13})\Zeta_{13}J_{23}(T_-)_3(\Zeta_{13})\Zeta_{13}^{-1}(\Omega_L)_{13}^{-1}.
\]
By Lemma~\ref{lem:516} (ii) this final equation reduces to the desired equality 
\[
A_R^3(J_{1,23}J_{23})_0=B_{1,23}J_{23}=(J_{1,23}J_{23})_0.
\]
This implies that $J_{1,23}J_{23}$ solves the right $3$-component ABRR equation, and hence the mixed ABRR equation $A_L^3A_R^3(J_{1,23}J_{23})=J_{1,23}J_{23}$.
\end{proof}

\begin{proof}[Proof of Theorem~\ref{thm:twists}]
By uniqueness of solutions to the mixed ABRR equation, Proposition~\ref{prop:mixedABRR}, and~\eqref{eq:492}, we see that $J_{12,3}J_{12}=J_{1,23}J_{23}$.  The remaning identity $(\epsilon\ox 1)(J)=(1\ox\epsilon)(J)=1$ follows from the identity $(\epsilon\ox 1)(R)=(1\ox\epsilon)(R)=1$ and the fact that $\epsilon$ commutes with $T_{\pm}$.
\end{proof}

\section{Subalgebras from the $R$-matrix}
\label{sect:radford}

We recall here some information from~\cite{radford93}.  We will let $D(H)$ denote the Drinfeld double of a Hopf algebra $H$.  Recall that this is a quasitriangular Hopf algebra which, as a coalgebra, is simply the tensor coalgebra $D(H)=H\ox (H^\ast)^{cop}$.  Recall also that the two inclusions $H\to D(H)$ and $(H^\ast)^{cop}\to D(H)$ are Hopf algebra maps.  This is all the information we will need about the Drinfeld double, and we invite the reader to see~\cite[Sect. 10.3]{montgomery} for more information.

\subsection{The right and left subalgebras from $R$}
\label{sect:lrsubalg}

Let $H=(H,R)$ be a quasitriangular Hopf algebra.  We can consider for any $Q\in H\ox H$ the functions $t_Q:H^\ast\to H$, $f\mapsto (f\ox 1)(Q)$ and $t'_Q:H^\ast\to H$, $f\mapsto (1\ox f)(Q)$.  Indeed, for any $H_1,H_2\subset H$ with $Q\in H_1\ox H_2$ we can restrict these functions to $t_Q:H_1^\ast\to H_2$, $t'_Q:H_2^\ast\to H_1$.  For the $R$-matrix we have the right and left subspaces in $H$ defined as follows.

\begin{definition}
For a quasitriangular Hopf algebra $H=(H,R)$ we take $R_{(r)}=t_R(H^\ast)$ and $R_{(l)}=t'_R(H^\ast)$.
\end{definition}

We may refer to these subalgebras as the Radford subalgebras associated to $R$.  Using the properties of the $R$-matrix one shows

\begin{proposition}[{\cite[Prop. 2]{radford93}}]
The subspaces $R_{(l)}$ and $R_{(r)}$ are Hopf subalgebras in $H$.  Furthermore, the maps $t_R$ and $t'_{R}$ provide Hopf morphisms $(H^\ast)^{cop}\to H$ and $(H^\ast)^{op}\to H$, and Hopf isomorphisms $(R_{(l)}^\ast)^{cop}\overset{\cong}\to R_{(r)}$ and $(R_{(r)}^\ast)^{op}\overset{\cong}\to R_{(l)}$. 
\end{proposition}

Take $H_R=R_{(l)}R_{(r)}$.  It turns out that this is a Hopf subalgebra in $H$, and that it is the minimal Hopf subalgebra in $H$ with $R\in H_R\ox H_R$.  A quasitriangular Hopf algebra is called {\it minimal} if $H=H_R$.  Strictly speaking, we will not be needing the following result.  It does, however, inform the approach of the current work, and so we repeat it here.

\begin{theorem}[{\cite[Thm. 2]{radford93}}]\label{thm:rad2}
For a minimal Hopf algebra $H$ there is a (unique) surjective map of quasitriangular Hopf algebras $Y:D(R_{(l)})\to H$ with $Y|R_{(r)}$ the inclusion and $Y|(R_{(l)}^\ast)^{cop}=t_R$.
\end{theorem}

Taking the dual of $Y$, we see that there is a algebra inclusion
\[
H^\ast\to R_{(l)}^\ast\ox R_{(l)}^{op}\cong R_{(r)}\ox R_{(l)}^{op}
\]
given as the composite $H^\ast\overset{\Delta}\to H^\ast\ox H^\ast\overset{t_R\ox t'_R}\longrightarrow R_{(r)}\ox R_{(l)}^{op}$.
\par

We note that although minimality is not preserved under twists, a stronger condition called factorizability is preserved under twists.  Indeed, a finite dimensional quasitriangular Hopf algebra $H$ is factorizable if and only if the M\"uger center of $\mathrm{rep}(H)$ is trivial~\cite{EGNO}.  Small quantum groups are examples of factorizable Hopf algebras, and so the twists $u_q^J$ will be factorizable and hence minimal.

\subsection{Bicharacters as twists on abelian groups}
\label{sect:bchar}

\begin{lemma}
Let $\Lambda$ be a finite abelian group.  Any bicharacter $B\in \C[\Lambda]\ox \C[\Lambda]$ is an $R$-matrix for $\C[\Lambda]$.
\end{lemma}

\begin{proof}
We need to check the equations $(\Delta\ox 1)(B)=B_{13}B_{23}$, $(1\ox \Delta)(B)=B_{13}B_{12}$, and $B\Delta(\lambda)B^{-1}=\Delta(\lambda)$ for each $\lambda\in \Lambda$.  The first two equations follow from the fact that $B$ is a bicharacter, and the final equation follows from the fact that $\Lambda$ is abelian.
\end{proof}

In the case of a bicharacter $B$ giving an $R$-matrix for $\C[\Lambda]$, the two maps $t_B$ and $t'_B$ restrict to, and are specified by, the standard group maps $\Lambda^\vee\to \Lambda$ induced by $B$.

\begin{definition}
Given a bicharacter $B\in\C[\Lambda]\ox\C[\Lambda]$, for $\Lambda$ a finite abelian group, let $\Lambda_{(r)}$ and $\Lambda_{(l)}$ denote the images $t_B(\Lambda^\vee)$ and $t'_B(\Lambda^\vee)$ in $\Lambda$ respectively.
\end{definition}

We have $B_{(r)}=\C[\Lambda_{(r)}]$ and $B_{(l)}=\C[\Lambda_{(l)}]$.

\section{Parabolic subalgebras in $u_q(\g)^J$ and Radford's subalgebras}
\label{sect:radford2}

For this section fix a Belavin-Drinfeld triple $(\Gamma_1,\Gamma_2,T)$ and solution $\S$ to~\eqref{eq:seq}.  Fix also $J=J_{T,\S}$ from Theorem~\ref{thm:twists}.
\par

We saw in Section~\ref{sect:lrsubalg} that there are algebra surjections $t_{R^J}:(u_q^J)^\ast\to R_{(r)}^J$ and $t'_{R^J}:(u^J_q)^\ast\to (R^J_{(l)})^{op}$.  Our main goal is to show that the map
\[
\mathrm{Irrep}(R_{(r)}^J)\to \mathrm{Irrep}\left((u_q^J)^\ast\right)
\]
induced by restriction is a bijection, modulo the action of a finite character group.  Both to understand the irreps of $R_{(r)}^J$ and to establish this proposed bijection we need to understand the subalgebra $R_{(r)}^J$, which we study here.
\par

For any root $\alpha$ we will take $\bar{\alpha}$ to be the component of $\alpha$ in $\L$, under the decomposition $\G=\L\times\L^\perp$.

\subsection{A preemptive change of coordinates}
Let us take
\[
\E_{\alpha}=q^{-\frac{1}{4}(\bar{\alpha},\bar{\alpha})}K^{1/2}_{\bar{\alpha}} E_\alpha\ \ \mathrm{and} \ \ \F_{\beta}=q^{-\frac{1}{4}(\bar{\beta},\bar{\beta})}K^{-1/2}_{\bar{\beta}}F_\beta.
\]
These new generators satisfy the appropriate relations so that we have an algebra automorphism
\[
\text{change of coord's}:u_q\overset{\cong}\to u_q,\ \ \left\{\begin{array}{l}
E_\alpha\mapsto \E_\alpha\\
F_{\beta}\mapsto \F_{\beta}\\
K_{\mu}\mapsto K_\mu.\end{array}\right.
\]
\par

Recall that each $E_\mu$, $\mu\in\Phi^+$, is a linear combination of permutations of the monomial $E_{\alpha_{i_1}}\dots E_{\alpha_{i_m}}$, where $\mu=\alpha_{i_1}+\dots+\alpha_{i_m}$ with the $\alpha_{i_k}$ simple.  So each $E_\mu$ is sent to $q^{-\frac{1}{4}(\bar{\mu},\bar{\mu})}K^{1/2}_{\bar{\mu}} E_\mu$ under the above change of coordinates.  A similar statement holds for the $F_\nu$, and we may adopt a consistent notation
\[
\E_\mu=q^{-\frac{1}{4}(\bar{\mu},\bar{\mu})}K^{1/2}_{\bar{\mu}} E_\mu\ \ \mathrm{and}\ \ \F_\nu=q^{-\frac{1}{4}(\bar{\nu},\bar{\nu})}K^{-1/2}_{\bar{\nu}}F_\nu,
\]
for $\mu,\nu\in\Phi^+$.  These bold elements produce a $\C[G]$-basis for $u_q=u_q^J$ just as in Lemma~\ref{thm:basis}.

\subsection{The quantum parabolics in $u_q(\g)^J$ and the $R$-matrix}

For a fixed subset $\Sigma\subset \Gamma$ we let $\p_+=\p_+(\Sigma)$ denote the corresponding positive parabolic in $\g$ and $u_q(\p_+)$ denote the Hopf subalgebra
\[
u_q(\p_+)=\C\langle G,\E_{\alpha},\F_\beta:\alpha\in \Gamma,\ \beta\in \Sigma\rangle\ \subset\ u_q(\g).
\]
We have the negative analog
\[
u_q(\p_-)=\C\langle G,\E_{\beta},\F_\alpha:\beta\in \Sigma, \alpha\in \Gamma\rangle\ \subset\ u_q(\g).
\]
We let $u_q(\p^{ss})$ denote the small quantum group associated to the (union of) Dynkin diagram(s) $\Sigma$ in $\Gamma$, and suppose that the perpendicular $G_{\Sigma}^\perp$ to the subgroup $G_{\Sigma}=\mathbb{Z}/l\mathbb{Z}\cdot \{K_{\beta}:\beta\in\Sigma\}$ in $G$ is a complement to $G_{\Sigma}$.

\begin{lemma}\label{lem:parbol}
Let $\Sigma$ be a subset in $\Gamma$ and $\p$ denote the corresponding positive (resp. negative) parabolic.  There is an algebra surjection
\begin{equation}\label{eq:710}
u_q(\p)\to \C[G_{\Sigma}^\perp]\ox u_q(\p^{ss}),\ \ \left\{\begin{array}{ll}
\E_\beta\mapsto \E_\beta &\mathrm{when\ }\beta\in\Sigma\\
\F_\beta\mapsto \F_\beta &\mathrm{when\ }\beta\in\Sigma\\
\E_\alpha\ (\text{resp. }\F_\alpha)\mapsto 0 &\mathrm{when\ }\alpha\in \Gamma-\Sigma\\
K_\gamma\mapsto K_\gamma
\end{array}\right.
\end{equation}
with kernel equal to the nilpotent ideal $N=(\E_\alpha:\alpha\in \Gamma-\Sigma)$ (resp. $N'=(\F_\alpha:\alpha\in\Gamma-\Sigma)$).
\end{lemma}

\begin{proof}
It suffices to prove the result for the positive parabolic.  We arrive at the result for the negative quantum parabolic by considering the automorphism of $u_q(\g)$ which exchanges the $E_\alpha$ and $F_\alpha$, and inverts the $K_\gamma$, and hence exchanges the positive and negative parabolics.  Simply by checking relations we see that there is a surjective algebra map $\C[G_{\Sigma}^\perp]\ox u_q(\p^{ss})\to u_q(\p)/N$ defined on the generators in the obvious way.  We will show that this map is injective by counting dimensions.
\par

We have that the nonnegative part of $u_q(\p)$ is all of $u_+$, and we see for grading reasons that $u_q(\p)_-$ is free over $\C[G]$ with basis given by ordered monomials in the $\F_{\nu}$ with $\nu$ a positive root in the $\mathbb{Z}$-span on $\Sigma$ (see Theorem~\ref{thm:basis}).  By the commutativity relation between the $\E$ and $\F$ we see that the restriction of the multiplication map $\theta:u(\p)_+\ox_{\C[G]}u(\p)_-\to u_q(\p)$ is surjective.  Since this map is given by restricting the isomorphism $u_q(\g)_+\ox_{\C[G]} u_q(\g)_-\to u_q(\g)$, and since all modules are flat over $\C[G]$, we see that $\theta$ is injective as well.  It follows that $u_q(\p)$ has the obvious basis consisting of ordered monomials in the $\E_\mu$ and $\F_\nu$, where $\nu$ is as above.
\par

When we take the quotient we now see that $u_q(\p)/N$ has a $\C[G]$-basis of orderend monomials in the $\E_{\mu'}$ and $\F_{\nu}$, with $\mu',\nu\in\Phi^+\cap(\mathbb{Z}\cdot \Sigma)$.  Since $\Psi=\Phi\cap (\mathbb{Z}\cdot \Sigma)$ is the root system for $\p^{ss}$, we find by Lusztig's basis for $u_q(\p^{ss})$ that $u_q(\p)/N$ and $\C[G_{\Sigma}^\perp]\ox u_q(\p^{ss})$ have the same dimension.  Whence our surjection is an isomorphism.  The inverse is given by the same formulas as~\eqref{eq:710}, and implies the existence of~\eqref{eq:710}.
\par

As for nilpotence of $N$, when we grade by the group $\mathbb{Z}\{\Gamma-\Sigma\}$ we see that $N^k$ lay in degrees $\mathbb{Z}_{\geq k}\{\Gamma-\Sigma\}$.  Since $u_q(\p)$ is finite dimensional it has no nonzero elements in degrees $\mathbb{Z}_{\geq k}\{\Gamma-\Sigma\}$ for large $k$.
\end{proof}

\subsection{Quantum parabolics and BD triples}

\begin{definition}
For any Belavin-Drinfeld triple $(\Gamma_1,\Gamma_2,T)$ we let $u_q(\p_1)$ and $u_q(\p_2)$ denote the positive and negative quantum parabolics in $u_q(\g)$ corresponding to $\Gamma_1$ and $\Gamma_2$ respectively.
\end{definition}

So $u_q(\p_1)$ contains only the $\F_\alpha$ with $\alpha\in \Gamma_1$, and $u_q(\p_2)$ contains only the $\E_\beta$ with $\beta\in \Gamma_2$.  Recall our twist $J=J_{T,\S}$ and the definition $R^J=J_{21}^{-1}RJ$.  We have also
\begin{equation}\label{eq:J21-1}
\begin{array}{l}
J_{21}^{-1}\\
=(1\ox T)(\Omega)\dots (1\ox T^n)(\Omega)\S^{-1}\Omega_{L^\perp}^{1/2}(S^{-1}\ox T^n_+)(R_{21})\dots(S^{-1}\ox T_+)(R_{21})\\
=(T^{-1}\ox 1)(\Omega)\dots (T^{-n}\ox 1)(\Omega)\S^{-1}\Omega_{L^\perp}^{1/2}(T_-^n\ox S)(R_{21})\dots(T_-\ox S)(R_{21}).
\end{array}
\end{equation}

\begin{lemma}
There are containments $R^J_{(r)}\subset u_q(\p_2)$ and $R^J_{(l)}\subset u_q(\p_2)$.
\end{lemma}

\begin{proof}
This is immediate from the form of $J$ and $R$, and the fact that $(T_+^k\ox 1)(R)=(1\ox T_-^k)(R)$.
\end{proof}

We consider $\C[G]$ as a quasitriangular Hopf algebra with $R$-matrix $\Omega$.  Then $\S^{-1}$ provides a twist for $\C[G]$ and the new $R$-matrix $\S_{21}\Omega\S^{-1}=\S^{-2}\Omega$.  We take $G_{(r)}$ and $G_{(l)}$ the right and left subgroups associated to $\S^{-2}\Omega$, as in Section~\ref{sect:bchar}.  Note that by the duality $t_{\Omega^{\S^{-1}}}:G_{(l)}^\vee\overset{\cong}\to G_{(r)}$ we know that these two groups have the same order.  We want to prove

\begin{proposition}\label{prop:R}
The inclusions $R^J_{(r)}\subset u_q(\p_2)$ and $R^J_{(l)}\subset u_q(\p_1)$ are equalities exactly when $G_{(r)}=G_{(l)}=G$.  In general, we have that
\[
R^J_{(r)}=\C\langle G_{(r)},\E_\beta,\F_\gamma:\beta\in \Gamma_2,\gamma\in \Gamma\rangle
\]
and
\[
R^J_{(l)}=\C\langle G_{(l)},\E_\gamma,\F_\alpha:\alpha\in \Gamma_1,\gamma\in \Gamma\rangle
\]
in $u_q(\g)$.
\end{proposition}

Section~\ref{sect:proof} is dedicated to a proof of Proposition~\ref{prop:R}.  As a corollary we will have

\begin{corollary}\label{cor:R}
Let $\N\subset R^J_{(r)}$ denote the preimage of the ideal $N=(\F_{\beta}:\beta\in\Gamma-\Gamma_2)$ in $u_q(\p_2)$ along the inclusion $R^J_{(r)}\to u_q(\p_2)$.  Take also $\Lambda=G_{(r)}\cap G_2^\perp$.  Then we have a canonical algebra isomorphism
\[
R^J_{(r)}/\N\overset{\cong}\to \C[\Lambda]\ox u_q(\p_2^{ss}).
\]
For the analogously defined $\N'\subset R^J_{(l)}$ and $\Lambda'\subset G_{(l)}$ we have also
\[
R^J_{(l)}/\N'\overset{\cong}\to \C[\Lambda']\ox u_q(\p_1^{ss}).
\]
\end{corollary}

\begin{proof}
The isomorphisms come from restricting the isomorphisms of Lemma~\ref{lem:parbol} along the inclusion $R^J_{(\ast)}\to u_q(\p_{\ast})$.
\end{proof}

\subsection{An example}

It seems, from considering examples, that the subalgebra $R_{(r)}^J$ will often be the full parabolic $u_q(\p_2)$.  This will always be the case, for example, when considering twists associated to maximal triples $(\Gamma_1,\Gamma_2,T)$ on $A_n$ (see the discussion following Lemma~\ref{lem:solex}).  To construct an example in which the containment $R^J_{(r)}\subset u_q(\p_2)$ is proper we need only construct an example for which the containment $G_{(r)}\subset G$ is proper.
\par

We claim that in the following example we will have $G_{(r)}\subsetneq G$ and $R^J_{(r)}\subsetneq u_q(\p_2)$: Take $l=3$ and consider the tiple on $A_3$
\[
\xymatrixrowsep{.03mm}
\xymatrix{
{\color{magenta}\alpha_1} & \\
\bullet\ar@{-}[r] &\alpha_2\ar@{-}[r] &\bullet\\
 & &{\color{blue}\alpha_3}
}
\]
with $\Gamma_1=\{\alpha_1\}$, $\Gamma_2=\{\alpha_3\}$, $T(\alpha_1)=\alpha_3$.  We have here
\[
\L=\mathbb{Z}/3\mathbb{Z}\cdot\{\alpha_2,\alpha_1+\alpha_3\},\ \ \G_2^\perp=\mathbb{Z}/3\mathbb{Z}\cdot\{\alpha_1,\alpha_2+\frac{1}{2}\alpha_3\}
\]
and the form on $\G_2^\perp$ is given (in multiplicative notation) by
\[
(\alpha_1,\alpha_1)=q^2,\ (\alpha_1,\alpha_2+\frac{1}{2}\alpha_3)=q^{-1},\ (\alpha_2+\frac{1}{2}\alpha_3,\alpha_2+\frac{1}{2}\alpha_3)=1.
\]
The (unique) solution $\S$ with $\S(\alpha_1+\alpha_3,\alpha_2)=q$ is such that
\[
\begin{array}{ll}
\S^{-2}\Omega(\alpha_1+\alpha_3,\alpha_1)=q, & \S^{-2}\Omega(\alpha_2,\alpha_1)=q^{-1}\\
\S^{-2}\Omega(\alpha_1+\alpha_2,\alpha_2+\frac{1}{2}\alpha_3)=q^{-1}, &
\S^{-2}\Omega(\alpha_2,\alpha_2+\frac{1}{2}\alpha_3)=q
\end{array}
\]
and hence
\[
t_{\S^{-2}\Omega}(\alpha_1+\alpha_3)=t_{\S^{-2}\Omega}(\alpha_2)^{-1}=K_{\alpha_1}K_{\alpha_2+\frac{1}{2}\alpha_3}\ \mathrm{mod}\ G_2.
\]
It follows that $G_{(r)}=\mathbb{Z}/3\mathbb{Z}\oplus G_2=(\mathbb{Z}/3\mathbb{Z})^2$ is a proper subgroup of $G=(\mathbb{Z}/3\mathbb{Z})^3$.  Alternatively, the solution with $\S(\alpha_1+\alpha_3,\alpha_2)=1$ yields $G_{(r)}=G$.

\section{A proof of Proposition~\ref{prop:R}}
\label{sect:proof}

Fix $J=J_{T,\S}$.  We will prove by direct calculation that $R^J_{(r)}$ is as proposed in Proposition~\ref{prop:R}.  Let $J'$ be the twist associated to the triple $\Gamma'_1=\Gamma_2$, $\Gamma'_2=\Gamma_1$, $T'=T^{-1}$, and solution $\S'=\S^{-1}$.  The result for $R^J_{(l)}$ can be subsequently be deduced from the fact that the algebra automorphism $\phi:u_q\to u_q$, which exchanges $E_\alpha$ with $F_\alpha$ and sends $K_\gamma$ to $K_{\gamma}^{-1}$, is such that $\phi(R^J_{(l)})=R^{J'}_{(r)}$.

\subsection{Some supporting results}

\begin{lemma}\label{lem:L1}
$G_{(r)}\subset R^J_{(r)}$ and $G_2\subset G_{(r)}$.
\end{lemma}

\begin{proof}
We have the $(u_-,u_+)$-bimodule isomorphism $u_-\ox_{\C[G]}u_+\to u_q$ given by multiplication.  The two projections $u_{\pm}\to \C[G]$ then give $u_-\ox_{\C[G]}u_+\to \C[G]$ and hence a bimodule projection $\Pi:u_q\to \C[G]$.  This gives an embedding $\Pi^\ast:\C[G]^\ast\to u_q^\ast$.  We have
\[
(\Pi\ox 1)(R^J)=\S^{-1}\Omega^{1/2}_{L^\perp}\Omega\S^{-1}\Omega^{-1/2}_{L^\perp}=\S^{-2}\Omega\in \C[G]\ox u_q.
\]
Note that $\S^{-2}\Omega=(\S_{21}^{-1})^{-1}\Omega\S^{-1}$ so that for any character $\mu\in \G$ we have
\[
(\mu\Pi\ox 1)(R^J)=t_{\Omega^{\S^{-1}}}(\mu)
\]
and hence $t_{R^J}(\G\Pi)=G_{(r)}$.  This gives the proposed inclusion $G_{(r)}\subset R^J_{(r)}$.
\par

As for the inclusion $G_2\subset G_{(r)}$ note that for $\alpha\in \Gamma_1$ we have
\[
\S^{-2}\Omega(\alpha-T\alpha,?)=\Omega^{-1}(\alpha+T\alpha,?)\Omega(\alpha-T\alpha,?)=\Omega^{-2}(T\alpha,?)=K^{-2}_{T\alpha}.
\]
Since $2$ is a unit in $\mathbb{Z}/l\mathbb{Z}$ we see that each $K_{T\alpha}\in G_{(r)}$ and hence $G_2\subset G_{(r)}$.
\end{proof}

The inclusion $G_2\subset G_{(r)}$ and splitting $G=G_2^\perp\times G_2$ implies that $G_{(r)}$ splits as $G_{(r)}=\Lambda\times G_2$, where $\Lambda=G_2^\perp\cap G_{(r)}$.
\par

In the following lemma we use the fact that for any bicharacter $B\in \C[G]\ox \C[G]$ we have
\[
B=\sum_{\mu,\nu\in\G}B(\mu,\nu)P_\mu\ox P_\nu,
\]
where $P_\mu=|G|^{-1}\sum_{\gamma\in \G}\mu(K_\gamma^{-1})K_\gamma$ is the idempotent associated to $\mu$.  Note that $P_\mu P_\nu=\delta_{\mu,\nu}P_\mu$ and $\mu(P_\nu)=\delta_{\mu,\nu}$.  For any bicharacter $B$ and $\mu\in \G$ we take
\[
B(\mu)=\text{the unique element in }\G\text{ with }B(\mu,\nu)=\Omega(B(\mu),\nu)\ \forall\ \nu\in \G.
\]

\begin{lemma}\label{lem:868}
For any bicharacter $B$, and $\alpha,\beta\in\Gamma$, we have
\[
(E_\alpha\ox F_\beta) B=B(K_{B_{21}(\beta)}\ox K_{B^{-1}(\alpha)})(E_\alpha\ox F_\beta)
\]
and
\[
(F_\beta\ox E_\alpha) B=B(K_{B_{21}^{-1}(\alpha)}\ox K_{B(\beta)})(F_\beta\ox E_\alpha).
\]
\end{lemma}

\begin{proof}
We have
\[
E_\alpha K_\gamma=q^{-(\alpha,\gamma)}K_\gamma E_\alpha\Rightarrow E_\alpha P_\mu=P_{\mu+\alpha}E_\alpha
\]
and $F_\beta P_\nu=P_{\nu-\beta}F_\beta$.  So for any bicharacter $B$ we have
\[
\begin{array}{l}
(E_\alpha\ox F_\beta) B=(\sum_{\mu,\nu} B(\mu,\nu)P_{\mu+\alpha}\ox P_{\nu-\beta})(E_\alpha\ox F_\beta)\\
=(\sum_{\mu,\nu} B(\mu-\alpha,\nu+\beta)P_{\mu}\ox P_{\nu})(E_\alpha\ox F_\beta)\\
=B(\sum_{\mu,\nu} B(\mu,\beta)B^{-1}(\alpha,\nu)P_{\mu}\ox P_{\nu})(E_\alpha\ox F_\beta)\\
=B(K_{B_{21}(\beta)}\ox K_{B^{-1}(\alpha)})E_\alpha\ox F_\beta.
\end{array}
\]
We arrive at the equation for $F_\beta\ox E_\alpha$ similarly.
\end{proof}

Considering the case $B=\Omega_L^{1/2}$, for each $\beta\in \Gamma_2$ we have
\[
\begin{array}{rl}
(E_\beta\ox F_{T^{-k}\beta}) \Omega^{1/2}_L & =\Omega^{1/2}_L(K_{\bar{T^{-k}\beta}}^{1/2}E_\beta\ox K_{\bar{\beta}}^{-1/2}F_{T^{-k}\beta})\\
& =\Omega^{1/2}_L(K_{\bar{\beta}}^{1/2}E_\beta\ox K_{\bar{\beta}}^{-1/2}F_{T^{-k}\beta})\\
& =q^{\frac{1}{2}(\bar{\beta},\bar{\beta})}\Omega^{1/2}_L(\E_\beta\ox \F_{T^{-k}\beta}).
\end{array}
\]
Similarly $(F_\alpha\ox E_{T^{k}\alpha}) \Omega^{1/2}_L=q^{\frac{1}{2}(\bar{\alpha},\bar{\alpha})}\Omega^{1/2}_L(\F_\alpha\ox \E_{T^k\alpha})$ for $\alpha\in \Gamma_1$.

\subsection{Proof of Proposition~\ref{prop:R}}

As explained in the beginning of the section, we need only prove the proposition for $R^J_{(r)}$.  We prove the proposition in two parts.  First we establish the containment $\C\langle G_{(r)},\E_{\alpha},\F_\beta:\alpha\in \Gamma_2,\beta\in\Gamma\rangle\subset R^J_{(r)}$, then we establish the opposite containment.

\begin{proof}[Proof of Proposition~\ref{prop:R}]
{\bf Part I:} Take
\[
\Omega(k, m) =\prod_{k\leq i\leq m}(T^i\ox 1)(\Omega)\ \ \mathrm{and}\ \ 
\Omega'(k,m)=\prod_{k\leq j\leq m}(1\ox T^j)(\Omega),
\]
with the empty product equal to $1$.  Now $J$ appears as
\[
(1\ox T_-)(R)\dots (1\ox T^n_-)(R)\S^{-1}\Omega_{L^\perp}^{-1/2}\Omega(1,n)^{-1}
\]
and $J_{21}^{-1}$ appears as
\[
\Omega'(1,n)\S^{-1}\Omega^{1/2}_{L^\perp}(S^{-1}\ox T^n_+)(R_{21})\dots (S^{-1}\ox T_+)(R_{21}).
\]
It suffices to prove that each of the $\E_{\alpha}$ and $\F_\beta$ are in $R^J_{(l)}$, by Lemma~\ref{lem:L1}.
\par

From our $\C[G]$-basis for $u_q$ we have the $\C[G]$-linear projection
\[
\pi^E_\beta:u_q\to \C[G]E_\beta
\]
which annihilates each of the basis elements from Theorem~\ref{thm:basis}, save for $E_\beta$.  More specifically, we take $\pi_\beta^E$ to be the obvious projection composed with the scaling by $q(1-q^2)^{-1}$.  Then we have
\[
\begin{array}{l}
(\pi^E_\beta\ox 1)(R^J)\\
=\sum_{k=0}^{m(\beta)}\S^{-1}\Omega^{1/2}_{L^\perp}\Omega(0,k-1)(E_\beta\ox F_{T^{-k}\beta}) \Omega(k,n)\S^{-1}\Omega_{L^\perp}^{-1/2}\Omega(1,n)^{-1}\\
=\sum_k\S^{-1}\Omega^{1/2}_{L^\perp}\Omega(0,k-1)( E_\beta\ox F_{T^{-k}\beta}) \Omega(1,k-1)^{-1}\S^{-1}\Omega_{L^\perp}^{-1/2}\\
=\sum_k\S^{-1}\Omega^{1/2}_{L^\perp}\Omega(0,k-1)(E_\beta\ox F_{T^{-k}\beta}) \Omega(0,k-1)^{-1}\S^{-1}\Omega^{1/2}\Omega_L^{1/2},
\end{array}
\]
where $m(\beta)=0$ when $\beta\notin \Gamma_2$ and otherwise $m(\beta)$ is minimal with $T^{-m(\beta)}\beta\notin \Gamma_2$ and $T^{-i}\beta\in\Gamma_2$ for $0\leq i<m(\beta)$.  We have
\[
\Omega(0,k-1)_{21}^{-1}(T^{-k}\beta)=-\sum_{i=1}^{k}T^{-i}\beta\ \ \mathrm{and}\ \  \Omega(0,k-1)(\alpha)=\sum_{j=0}^{k-1}T^j\beta
\]
so that the final expression reduces to
\[
\begin{array}{l}
\sum_k\S^{-1}\Omega^{1/2}_{L^\perp}(K_{\sum_{i=1}^kT^{-i}\beta}^{-1}\ox K_{\sum_{j=0}^{k-1}T^j\beta})(E_\beta\ox  F_{T^{-k}\beta}) \S^{-1}\Omega^{1/2}\Omega_L^{1/2}\\
=\sum_k\S^{-2}\Omega_{L^\perp}\Omega_L^{1/2}(K^{1/2}_{\S^2\Omega(T^{-k}\beta)}K_{\sum_iT^{-i}\beta}^{-1}\ox K^{-1/2}_{\S^{-2}\Omega(\beta)}K_{\sum_jT^j\beta})(E_\beta\ox  F_{T^{-k}\beta})\Omega_L^{1/2}\\
=q^{\frac{1}{2}(\bar{\beta},\bar{\beta})}\sum_k\S^{-2}\Omega(K^{1/2}_{\S^2\Omega(T^{-k}\beta)}K_{\sum_iT^{-i}\beta}^{-1}\ox K^{-1/2}_{\S^{-2}\Omega(\beta)}K_{\sum_jT^j\beta})(\E_\beta\ox  \F_{T^{-k}\beta}).
\end{array}
\]
For $\epsilon_{\beta}:\C[G]E_\beta\to \C$, $g E_\beta\mapsto q^{-\frac{1}{2}(\bar{\beta},\bar{\beta})}$, we then have
\begin{equation}\label{eq:976}
(\epsilon_\beta \pi^E_\beta\ox 1)(R^J)=\sum_{k=0}^{m(\beta)} K^{-1/2}_{\S^{-2}\Omega(\beta)}K_{\sum_{j=0}^{k-1}T^j\beta} \F_{T^{-k}\beta}\in R^J_{(r)}.
\end{equation}
Note that the coefficients $K^{-1/2}_{\S^{-2}\Omega(\beta)}K_{\sum_{j=0}^{k-1}T^j\beta}$ are all in $G_{(r)}$.
\par

When $m(\beta)=0$, i.e. when $\beta\in \Gamma-\Gamma_2$, the sum~\eqref{eq:976} is just the element $K^{-1/2}_{\S^{-2}\Omega(\beta)}\F_\beta$.  Since $K^{-1/2}_{\S^{-2}\Omega(\beta)}\in G_{(r)}\subset R^J_{(r)}$ this implies $\F_\beta\in R^J_{(r)}$.  Since $m(\beta)=m(T^{-1}\beta)+1$ when $\beta\in \Gamma_2$, it now follows from~\eqref{eq:976} and induction on $m(\beta)$ that all $\F_\beta\in R^J_{(r)}$.
\par

The computation for the $\E_\beta$, $\beta\in \Gamma_2$, is quite similar.  Namely, one show for $\alpha\in \Gamma_1$ that $(\pi^F_\alpha\ox 1)(R^J)$ is the a sum
\[
q^{\frac{1}{2}(\bar{\alpha},\bar{\alpha})}\sum_{k=1}^{m'(\alpha)}\S^{-2}\Omega(g_k\ox K^{1/2}_{\S^{-2}\Omega(\alpha)}K_{T^k\alpha+\sum_{j=1}^kT^j\alpha})(\F_\alpha\ox \E_{T^k\alpha}),
\]
where $\pi^F_\alpha$ is a scaling of the obvious projection and $g_k\in G$, then proceeds by induction on $m'(\alpha)$ just as above.

{\bf Part II:} We now give the opposite containment $R^J_{(r)}\subset \C\langle G_{(r)},\E_{\alpha},\F_\beta:\alpha\in \Gamma_2,\beta\in\Gamma\rangle$ to complete the proof.  We adopt the same notation $\Omega(k,m)$ and $\Omega'(k,m)$ as above.  We have that $R^J$ is a $\C[G]\ox \C[G_{(r)}]$-linear combination of elements of the form $M_1M_2M_3$ with
\[
M_1=\Omega'(1,n)\S^{-1}\Omega_{L^\perp}^{1/2}(F_{\xi_n}\ox E_{T^n\xi_n})(1\ox T^n)(\Omega)^{-1}\dots (F_{\xi_1}\ox E_{T\xi_1})(1\ox T)(\Omega)^{-1},
\]
$M_2=(E_{\zeta}\ox F_{\zeta})\Omega$,
\[
M_3=(E_{\eta_1}\ox F_{T^{-1}\eta_1})(T\ox 1)(\Omega)\dots (E_{\eta_n}\ox F_{T^{-n}\eta_n})\Omega(0,n-1)^{-1}\S^{-1}\Omega^{1/2}\Omega_L^{1/2}.
\]
Here the $\xi_k$ are in $\mathbb{Z}_{\geq 0}\Gamma_1$ with $T^i(\xi_k)\in \mathbb{Z}_{\geq 0}\Gamma_1$ for each $0\leq i<k$.  We take a similar restriction for the $\eta_j\in \mathbb{Z}_{\geq 0}\Gamma_2$ and let $\zeta$ be arbitrary in the positive root lattice.  For $\tau=\alpha_1+\dots +\alpha_m$ with the $\alpha_i$ simple roots, by $E_\tau$ (resp. $F_\tau$) we simply mean some permutation of the monomial $E_{\alpha_1}\dots E_{\alpha_n}$ (resp. $F_{\alpha_1}\dots F_{\alpha_n}$).  So we are deviating from the notation of Theorem~\ref{thm:basis} here.
\par

One simply moves all the bicharacters from the right to left, in order, using Lemma~\ref{lem:868}, to find
\[
\begin{array}{l}
M_1M_2M_3\\
=q^{\epsilon}(g_1\ox g_2)\S^{-2}\Omega\left((\prod_i \F_{\xi_i})\E_\zeta(\prod_j\E_{\eta_j})\ox (\prod_i \E_{T^i\xi_i})\F_\zeta(\prod_j\F_{T^{-j}\eta_j})\right)
\end{array}
\]
with $g_1\in G$ and $g_2\in G_{(r)}$.  Hence for any $f\in u_q^\ast$ we will have, for some constant $c_f\in \C$, 
\[
\begin{array}{l}
(f\ox 1)(M_1M_2M_3)\\
=c_fg_2t_{\S^{-2}\Omega}(f)(\prod_i \E_{T^i\xi_i})\F_\zeta(\prod_j\F_{T^{-j}\eta_j})\in \C\langle G_{(r)},\E_{\alpha},\F_\beta:\alpha\in \Gamma_2,\beta\in\Gamma\rangle.
\end{array}
\]
Since $R^J$ is a sum of such monomials $M_1M_2M_3$ we find
\[
R^J_{(r)}=t_{R^J}(u_q^\ast)\subset \C\langle G_{(r)},\E_{\alpha},\F_\beta:\alpha\in \Gamma_2,\beta\in\Gamma\rangle.
\]
\end{proof}

\section{Representation theory of the dual $(u_q(\g)^J)^\ast$}
\label{sect:reptheory}

In this section we describe the irreducible representations of the dual $(u_q(\g)^J)^\ast$,  for $J=J_{T,\S}$ as in Theorem~\ref{thm:twists}.

\subsection{Grouplikes and the parabolic subalgebras}

\begin{lemma}\label{lem:grouplike0}
Each $K_\mu\in L$ is grouplike in the twist $u_q^J$.
\end{lemma}

\begin{proof}
We claim that $K_\mu\ox K_\mu$ commutes with $J$, so that $\Delta^J(K_\mu)=J^{-1}(K_\mu\ox K_\mu)J=K_\mu\ox K_\mu$.  From the particular form on $J$, we see that it suffices to show that $K_\mu\ox K_\mu$ commutes with $T^k_+E_\nu\ox F_{\nu}$ and $F_{\nu}\ox T^k_+E_\nu$ for $\nu$ a positive root in $\mathbb{Z}\Gamma_1$ with $T^i\nu\in\mathbb{Z}\Gamma_1$ for all $0\leq i<k$.  But this is clear since $\nu-T^k\nu\in \L^\perp$ and $\mu\in \L$.
\end{proof}

Take $\heartsuit$ equal to either $(r)$ or $(l)$.  By Lemma~\ref{lem:grouplike0} we now see that the restriction of the multiplication map $\C[L]\ox R^J_{\heartsuit}\to u_q^J$ is a coalgebra map, where we just give $\C[L]$ its usual group ring structure.  If we let $L$ act on $R^J_{\heartsuit}$ by conjugation this gives a Hopf map $\C[L]\ltimes R^J_{\heartsuit}\to u_q^J$.  According to the particular form of $R^J_{\heartsuit}$ given in Proposition~\ref{prop:R}, and Lemma~\ref{lem:GT}, we see that this map has image equal to the corresponding quantum parabolic $u_q(\p_i)$.  So we find

\begin{lemma}
The quantum parabolics $u_q(\p_i)$ are both Hopf subalgebras in the twist $u_q(\g)^J$.
\end{lemma}

From the Hopf map $\C[L]\ltimes R^J_{(l)}\to u_q^J$ we also get a dual Hopf map
\begin{equation}\label{eq:map}
(u_q^J)^\ast\to \C[\L]\ox (R^J_{(l)})^\ast\overset{1\ox t_{R^J}}\to \C[\L]\ox R^J_{(r)}
\end{equation}
which extends to an algebra map 
\[
(u_q^J)^\ast\to \C[\L]\ox R^J_{(r)}/\N\cong \C[\L]\ox \C[\Lambda]\ox u_q(\p_2^{ss}),
\]
by Corollary~\ref{cor:R}.  (Recall our subgroup $\Lambda=G_{(r)}\cap G_2^\perp$ from Corollary~\ref{cor:R}.)  In a moment we will also need the following lemma.

\begin{lemma}\label{lem:L6}
Take $\mathscr{C}=L/(G_{(l)}\cap L)$.  The subgroup $\Lambda$ in $G_{(r)}$ is isomorphic to the dual $(G_{(l)}\cap L)^\vee$, and we have an exact sequence $0\to \mathscr{C}^\vee\to \L\to \Lambda\to 0$.
\end{lemma}

\begin{proof}
The dual $\Lambda^\vee$ gives the character group of $(R^J_{(l)})^\ast$, by Corollary~\ref{cor:R}.  The character group is identified with the group of grouplikes in $R^J_{(l)}$.  Since the intersection $G_{(l)}\cap L$ provides exactly $|\Lambda|=|G_{(r)}/G_2|=|G_{(l)}/G_1|$ grouplike elements in $R^J_{(l)}$ we see that $\Lambda=(G_{(l)}\cap L)^\vee$.  Whence we have an exact sequence $0\to \mathscr{C}^\vee\to \L\to \Lambda\to 0$.
\end{proof}

\subsection{Irreducible representations of $(u^J_q)^\ast$}

We take $\p^{ss}$ to be either of the (isomorphic) Lie algebras $\p_1^{ss}$ or $\p_2^{ss}$.  In this section we prove

\begin{theorem}\label{thm:dual}
There is a bijection
\[
\mathrm{Irrep}\left(\C[\L]\ox u_q(\p^{ss})\right)\overset{\cong}\to \mathrm{Irrep}\left((u_q^J)^\ast\right)
\]
given by restricting along an algebra surjection $(u_q^J)^\ast\to \C[\L]\ox u_q(\p^{ss})$.
\end{theorem}

\begin{remark}
In Theorem~\ref{thm:dual} we take advantage of the existence of an abstract algebra isomorphism $\C[\mathscr{C}^\vee\times \Lambda]\cong \C[\L]$, where $\mathscr{C}$ is as in Lemma~\ref{lem:L6}.  Such an isomorphism exists simply because both groups are abelian of the same order, by Lemma~\ref{lem:L6}.  However, as we'll see below, the character group of the dual $(u_q^J)^\ast$ is naturally identified with $L$, so that the appearance of $\L$ is appropriate.
\end{remark}

Before giving the proof we establish some background material.

\begin{lemma}\label{lem:L4}
The subcoalgebra $A$ in $R^J_{(l)}$ dual to the quotient $R^J_{(r)}/\N$, under the Hopf isomorphism $t_{R^J}:(R^J_{(l)})^\ast\to R^J_{(r)}$, is exactly the subalgebra $\C\langle G_{(l)},\E_\alpha,\F_\beta:\alpha\in \Gamma_2,\ \beta\in \Gamma_1\rangle$.
\end{lemma}

From the statement it is clear that $A$ is actually a Hopf subalgebra.  We are claiming that $A$ is the minimal subspace in $R^J_{(l)}$ admitting a factoring $(R^J_{(l)})^\ast\to A^\ast\to R^J_{(r)}/\N$.

\begin{proof}
Recall $\N$ is generated by all the $\F_\alpha$ with $\alpha\in \Gamma-\Gamma_2$.  For $\pi$ the projection $R^J_{(r)}\to R^J_{(r)}/\N$, one sees directly from the form of $R^J$ that $(1\ox \pi)(R^J)$ lay in the product $A'\ox (R^J_{(r)}/\N)$ where $A'=\C\langle G,\E_\alpha,\F_\beta:\alpha\in \Gamma_2,\ \beta\in \Gamma_1\rangle$.  But $(1\ox \pi)(R^J)$ also lay in $R^J_{(l)}\ox (R^J_{(r)}/\N)$ so that
\[
(1\ox \pi)(R^J)\in \left(R^J_{(l)}\ox (R^J_{(r)}/\N)\right)\cap \left(A'\ox (R^J_{(r)}/\N)\right).
\]
By flatness of everything over $\C$, this intersection is exactly $A\ox (R^J_{(r)}/\N)$.  So the surjective map $(R^J_{(l)})^\ast\to R^J_{(r)}/\N$ factors through $A^\ast$.  Since the dimensions of $A$ and $R^J_{(r)}/\N$ agree we must have that $A$ is in fact dual to $R^J_{(r)}/\N$.
\end{proof} 

The Hopf subalgebra $A$ is strongly related to the intersection of the quantum parabolics $u_q(\p_1)\cap u_q(\p_2)$, which we denote $\mathtt{Int}$.  From considering bases of the two quantum parabolics, as in Theorem~\ref{thm:basis}, one arrives at the presentation
\[
\mathtt{Int}=u_q(\p_1)\cap u_q(\p_2)=\C\langle G,\E_\alpha,\F_\beta:\alpha\in \Gamma_2,\ \beta\in \Gamma_1\rangle.
\]
Note that since the quantum parabolics are Hopf subalgebras, the intersection will be a Hopf subalgebra as well.
 
\begin{lemma}\label{lem:L5}
Take $\mathscr{C}=L/(G_{(l)}\cap L)$.  There is a coalgebra isomorphism $\C[\mathscr{C}]\ox A\to \mathtt{Int}$ given by multiplication.
\end{lemma}

\begin{proof}
We have the multiplication map $\C[L]\ox A\to u_q^J$ which is a surjection onto the intersection $\mathtt{Int}$.  Choose for each $\bar{\xi}\in \mathscr{C}$ a representative $\xi\in L$ and restrict the above multiplication map to get an coalgebra embedding
\[
\C[\mathscr{C}]\ox A=\bigoplus_{\bar{\xi}\in\mathscr{C}}\C\bar{\xi}\ox A\to u_q^J,\ \ \bar{\xi}\ox a\mapsto \xi\cdot a,
\]
with image exactly $\mathtt{Int}$.
\end{proof}

We have now the

\begin{proof}[Proof of Theorem~\ref{thm:dual}]
Let $K$ be the kernel of the projection $(u_q^J)^\ast\to \mathtt{Int}^\ast$ dual to the inclusion $\mathtt{Int}\to u_q^J$.  Note that $K$ will be a Hopf ideal in the dual.  We have the Hopf maps $(u_q^J)^\ast\to \C[\L]\ox R^J_{\heartsuit}$ of~\eqref{eq:map} which factor
\[
(u_q^J)^\ast\overset{\Delta}\to (u_q^J)^\ast\ox (u_q^J)^\ast\overset{(?)|L\ox t_{R^J}}\longrightarrow \C[\L]\ox R^J_{\heartsuit}.
\]
We claim that the induced maps
\[
F:(u_q^J)^\ast\to \C[\L]\ox R^J_{(r)}/\N\ \ \mathrm{and}\ \ F':(u_q^J)^\ast\to \C[\L]\ox R^J_{(l)}/\N'
\]
factor through $\mathtt{Int}^\ast$.  Equivalently, we claim that $K$ is in their kernels.  Let $\pi$ and $\pi'$ be the projections $\pi:R^J_{(r)}\to R^J_{(r)}/\N$ and $\pi':R^J_{(l)}\to R^J_{(l)}/\N'$.
\par

We prove the result for $R^J_{(r)}$.  Recall that $\N$ is the ideal generated by all the $\F_\alpha$ with $\alpha\in \Gamma-\Gamma_2$, and that $K$ consists of all functions vanishing on $\mathtt{Int}$.  Note that the intersection contains all of $\C[G]$, so that $K|L=0$.  Hence for each $f\in K$ we have
\[
F(f)=\sum_i (f_{i_1}|L)\ox \pi t_{R^J}(f_{i_2})=\sum_i (f_{i_1}|L)\ox \left((f_{i_2}\ox \pi)(R^J)\right)
\]
for some $f_{i_2}\in K$.  So it suffices to show $(K\ox \pi)(R^J)=0$.  However, we have already seen in Lemma~\ref{lem:L4} that $(1\ox \pi)(R^J)$ lay in $A\ox R^J_{(r)}/\N$, and $A\subset \mathtt{Int}$.  Hence $(f\ox \pi)(R^J)=0$ for each $f\in K$, and we find $(K\ox \pi)(R^J)=0$.  So the map $F$ factors through $\mathtt{Int}^\ast$, and a completely analogous argument shows that $F'$ factors through $\mathtt{Int}^\ast$ as well.
\par

Since $(u_q^J)^\ast\to\mathtt{Int}^\ast$ is a Hopf map the factorizations of $F$ and $F'$ imply that the map
\begin{equation}\label{eq:1221}
(u^J_q)^\ast\overset{\Delta}\to (u_q^J)^\ast\ox (u_q^J)^\ast\overset{F\ox F'}\to (\C[\L]\ox R^J_{(r)}/\N)\ox (\C[\L]\ox R^J_{(l)}/\N')
\end{equation}
factors
\begin{equation}\label{eq:1225}
(u^J_q)^\ast\to\mathtt{Int}^\ast\to (\C[\L]\ox R^J_{(r)}/\N)\ox (\C[\L]\ox R^J_{(l)}/\N').
\end{equation}
We note that that the map
\[
(u^J_q)^\ast\overset{\Delta}\to (u_q^J)^\ast\ox (u_q^J)^\ast\to (\C[\L]\ox R^J_{(r)})\ox (\C[\L]\ox R^J_{(l)})
\]
is an embedding, since its dual is a surjection, so that the kernels of~\eqref{eq:1221} and~\eqref{eq:1225} are nilpotent.  It follows that the kernel of the projection $(u^J_q)^\ast\to\mathtt{Int}^\ast$ must be nilpotent as well.
\par

We have from Lemmas~\ref{lem:L4} and~\ref{lem:L5} that $\mathtt{Int}^\ast\cong \C[\mathscr{C}^\vee]\ox R^J_{(r)}/\N$.  Recall from Corollary~\ref{cor:R} that $R^J_{(r)}/\N$ is isomorphic to $\C[\Lambda]\ox u_q(\p^{ss})$ and that $\C[\mathscr{C}^\vee\times\Lambda]\cong \C[\L]$, abstractly, to arrive at a surjection $(u^J_q)^\ast\to\C[\L]\ox u_q(\p^{ss})$ with nilpotent kernel.  Restricting then gives the proposed bijection on irreducible representations.
\end{proof}

\begin{corollary}\label{cor:grouplikes}
The set of grouplikes $G(u_q^J)$ is exactly $L$.
\end{corollary}

\begin{proof}
Since all the elements in $L$ are grouplike, by Lemma~\ref{lem:grouplike0}, we need only know that $|G(u_q^J)|=|L|$.  But this just follows from the theorem, since grouplikes in $u_q^J$ are identified with one dimensional representations of $(u_q^J)^\ast$.
\end{proof}

To compare $u_q$ to $u_q^J$ let us consider a maximal BD triple on $A_n$.  In this case there is a unique solution $\S$ to~\eqref{eq:seq} and, as mentioned previously, the algebras $R^J_{\heartsuit}$ will be the full parabolics.\footnote{This basically follows from the fact that $\G_2^\perp$ will be a free $\mathbb{Z}/l\mathbb{Z}$-module so that $\S(\mu,\nu)=0$ for any $\mu,\nu\in \G_2^\perp$, by antisymmetry.  Thus $\S^{-2}\Omega|\G_2^\perp\times\G_2^\perp=\Omega_{\G_2^\perp}$ and we must have all of $G_2^\perp$ in $G_{(r)}$.}  We will have, for $n=2$ and $l=5$ for example, a following variation in the dimensions of the coradicals:
\[
\dim\mathrm{Corad}(u(\sl_3))=25,\ \ \dim\mathrm{Corad}(u(\sl_3)^J)=105.
\]
The difference is made more stark from the fact that corepresentation theory of $u_q(\sl_{n+1})$ is essentially trivial, at least when we restrict our attention to the irreducibles and fusion rule, while the corepresentation theory of $u_q(\sl_{n+1})^J$ should be at least as complicated as the representation theory of $u_q(\sl_n)$.

\section{The Drinfeld element and properties of the antipode}
\label{sect:u}

Here we discuss preservation of the Drinfeld element under twisting.  Basic information on the Drinfeld element in a quasitriangular Hopf algebra, and its relation the antipode, can be found in~\cite{montgomery,EGNO}.  We fix a Belavin-Drinfeld triple $(\Gamma_1,\Gamma_2,T)$, solution $\S$, and twist $J=J_{T,\S}$ of $u_q(\g)$.

Let $\rho\in \G$ be the sum $\rho=\sum_{\mu\in\Phi^+}\mu$.  Then we have $(\rho,\alpha)=2$ for each simple root $\alpha$~\cite[Sect. 10.2]{humphreys72}.  This gives $S^2=\mathrm{ad}_{K_\rho}$ and the Drinfeld element for $u_q(\g)$ thus factors $u=K_\rho v$, where $v$ is a central element with $\Delta(v)=(v\ox v)(R_{21}R)^{-1}$, i.e. a ribbon element.  Note that $\rho$ is in $\L$ as $(\rho,\alpha-T(\alpha))=2-2=0$ for each $\alpha\in\Gamma$.  So $K_\rho$ remains grouplike in the twist $u_q(\g)^J$.
\par

Recall that under an arbitrary twist $J$ of a quasitriangular Hopf algebra $H$ the Drinfeld element for $H^J$ is the product $u^J=Q_J^{-1}S(Q_J)u$ (see e.g.~\cite{etingofgelaki02}).  In our case this means that the Drinfeld element for $u_q(\g)^J$ is given by
\[
u^J=Q_J^{-1}S(Q_J)K_\rho v.
\]
Centrality of $v$ implies 
\[
\Delta^J(v)=(v\ox v)J^{-1}(R_{21}R)^{-1}J=(v\ox v)(R^J_{21}R^J)^{-1}.
\]
So $v$ is still a ribbon element for the twist, and $Q_J^{-1}S(Q_J)K_\rho$ is grouplike in the twist.  Since $K_\rho$ itself is grouplike we conclude that $Q_J^{-1}S(Q_J)$ is grouplike as well.

\begin{proposition}\label{prop:preu}
The twists $J=J_{T,\S}$ are such that $Q_J^{-1}S(Q_J)=1$.
\end{proposition}

In the proof of the proposition we employ what we call a $T$-grading on $u_q^J$.  We define this as any algebra $\mathbb{Z}$-grading with the following properties:
\begin{enumerate}
\item[(a)] $\C[G]$ is homogeneous of degree $0$.
\item[(b)] The $E_\alpha$ are of positive degree and the $F_\alpha$ are of negative degree with $\deg(F_\alpha)=-\deg(E_\alpha)$.
\item[(c)] $\deg(E_{T\alpha})>\deg(E_\alpha)$ for each $\alpha\in \Gamma_1$.
\end{enumerate}
It is easy to construct such a grading.  For example, one can construct the acyclic directed graph $\mathtt{Graph}(\Gamma,T)$ with vertices $\Gamma$ and an arrow from $\alpha$ to $T(\alpha)$ for each $\alpha\in \Gamma_1$.  One then takes
\[
\deg(E_\alpha)=-\deg(F_\alpha)=|\mathtt{Graph}_{\leq\alpha}|,
\]
where $\mathtt{Graph}_{\leq\alpha}$ is the collection of all vertices with a path to $\alpha$ in $\mathtt{Graph}(\Gamma,T)$, including $\alpha$.  Note that the antipode preserves degree under any $T$-grading.

\begin{proof}
Under any $T$-grading on $u_q$ we will have that $J$ and $J^{-1}$ both lay in nonnegative degree in $u_q\ox u_q$, where $\deg(a\ox b)=\deg(a)+\deg(b)$ for $a,b\in u_q$.  This is clear from the explicit forms of the twist and its inverse given at Theorem~\ref{thm:twists} and~\eqref{eq:J21-1}.  We have also $J_0=\S^{-1}\Omega^{-1/2}_{L^\perp}$ and $(J^{-1})_0=\S\Omega^{1/2}_{L^\perp}$.  It follows, from the expressions of $Q_J$ and $Q_J^{-1}$ given in Section~\ref{sect:twinfo}, that both $Q_J^{-1}$ and $S(Q_J)$ lay in nonnegative degree with
\[
(Q_J^{-1})_0=m(\S^{-1}\Omega^{-1/2}_{L^\perp}),\ \ S(Q_J)_0=m(\S\Omega^{1/2}_{L^\perp}),
\]
where $m$ is multiplication.  We have now $\left(Q_J^{-1}S(Q_J)\right)_0=(Q_J^{-1})_0S(Q_J)_0$ and since the multiplication map on any commutative algebra, such as $\C[G]$, is a ring map
\[
(Q_J^{-1})_0S(Q_J)_0=m(\S^{-1}\Omega^{-1/2}_{L^\perp}\S\Omega^{1/2}_{L^\perp})=1.
\]
Finally we note that since $Q_J^{-1}S(Q_J)$ is grouplike it must lay in degree $0$.  Therefore $Q_J^{-1}S(Q_J)=\left(Q_J^{-1}S(Q_J)\right)_0=1$.
\end{proof}

As an immediate corollary we have

\begin{corollary}\label{cor:u}
The Drinfeld element for $u_q(\g)^J$ is equal to the Drinfeld element for $u_q(\g)$.
\end{corollary}

\subsection{Implications for the antipode}

In~\cite{negronng} the question was posed as to whether or not the order of the antipode and the traces of the powers of the antipode are preserved under twisting.  The question was answered positively for Hopf algebras with the Chevalley property.  Using the expression of the Chevalley property given in~\cite[Prop. 4.2, 5]{aeg} it is relatively easy to see that no small quantum group has the Chevalley property.  We can, however, verify the proposed invariance for Belavin-Drinfeld twists.

\begin{corollary}\label{cor:S}
For $S$ the antipode on $u_q(\g)$ and $S_J$ the antipode on the twist $u_q(\g)^J$, and $J=J_{T,\S}$, we have $\mathrm{Tr}(S_J^m)=\mathrm{Tr}(S^m)$ for all $m\in\mathbb{Z}$ and $\mathrm{ord}(S_J)=\mathrm{ord}(S)$.
\end{corollary}

\begin{proof}
Since $Q_J^{-1}S(Q_J)=1$ the proof of~\cite[Thm. 4.3]{negronng} still works to get $\mathrm{Tr}(S_J^m)=\mathrm{Tr}(S^m)$.  Since $S$ and $S_J$ are semisimple operators invariance of order follows from invariance of the traces.
\end{proof}

We can also get invariance of the so-called regular object of~\cite[Sect. 5.4]{shimizu15} using the condition of~\cite[Prop. 7.3 (ii)]{negronng}.  This positively answers~\cite[Question (5.12)]{shimizu15} for the twists $J_{T,\S}$ on small quantum groups.

\section{Twisted automorphisms and group actions on $\rep(u_q(\g))$}
\label{sect:G}

We use below the notion of a $2$-group.  A $2$-group is simply a monoidal category in which all morphisms are invertible and all objects have a weak inverse, i.e. an inverse up to isomorphism.  For a tensor category $\mathscr{C}$ we let $\underline{\mathrm{Aut}}(\mathscr{C})$ denote the $2$-group of autoequivalences of $\mathscr{C}$ as a tensor category, with natural isomorphisms, and $\mathrm{Aut}(\mathscr{C})$ denote the associated group of isoclasses of autoequivalences.
\par

Following Davydov~\cite{davydov10}, for a Hopf algebra $H$ we call a pair $(\phi,J)$ of a twist and a Hopf isomorphism $\phi:H\to H^J$ a {\it twisted automorphism} of $H$.  Each twisted automorphism can be identified with the tensor autoequivalence of $\rep(H)$ given by composing
\[
\rep(H)\overset{J}\to \rep(H^J)\overset{\mathrm{res}_\phi}\to \rep(H). 
\]
Indeed, twisted automorphisms form a $2$-subgroup in the $2$-group of autoequivalences $\underline{\mathrm{Aut}}(\rep(H))$ with product
$(\phi',J')\cdot (\phi,J)=(\phi\phi',J\phi^{\ox 2}(J'))$.  The induced isomorphisms between twisted automorphisms are gauge equivalences (see~\ref{sect:G2} below).  Furthermore, Ng and Schauenburg have shown that when $H$ is finite dimensional any autoequivalence of $\rep(H)$ will be isomorphic to a twisted automorphism~\cite[Thm. 2.2]{ngschauenburg08}.

We take $\g$ simple and simply laced, $q$ a primitive $l$th root of unity, for $l$ as in Section~\ref{sect:bdt}, and $u_q=u_q(\g)$.  In this final section we introduce twists $J^\lambda_\alpha$ of the small quantum group $u_q(\g)$ which are paired with automorphisms $\exp_\alpha^\lambda$ so that each pair $(\exp_\alpha^\lambda,J^\lambda_\alpha)$ provides a twisted automorphism of $u_q$.  We then relate a canonical algebraic group action on $\rep(u_q)$ to the twisted automorphisms $(\exp_\alpha^\lambda,J^\lambda_\alpha)$, and propose a question regarding a set of ``generators" for the collection of all twists of $u_q$.

\subsection{Twists via exponentiation: an extended quantum coadjoint action}

Recall that $u_q$ embeds as a Hopf subalgebra in Lusztig's divided powers quantum group
\[
U_q=U_q(\g)=\C\langle K^{\pm 1}_\alpha, E_\alpha,F_\alpha, E^{(l)}_\alpha,F^{(l)}_\alpha:\alpha\in \Gamma\rangle/(\mathrm{relations}).
\]
We do not recall the specific construction of $U_q$ here, and refer the reader instead to~\cite{lusztig90,lusztig89}, and in particular~\cite[Sect. 6.5]{lusztig90}, for the details.
\par

According to~\cite[Lem. 4.5]{lusztig89} the commutator
\[
\mathrm{ad}_{E^{(l)}_\alpha}:U_q\to U_q,\ \ x\mapsto [E^{(l)}_\alpha, x]
\]
preserves the subalgebra $u_q$ and the restriction $\mathrm{ad}_{E^{(l)}_\alpha}|u_q$ is a nilpotnent operator.  The same is true if we scale by any $\lambda\in \C$.  Hence we can exponentiate this operator to produce an algebra automorphism
\[
\exp^\lambda_\alpha:=\exp(\mathrm{ad}_{\lambda E^{(l)}_\alpha}|u_q)
\]
of the small quantum group $u_q$.  We can similarly define
\[
\exp^\lambda_{-\alpha}:=\exp(\mathrm{ad}_{\lambda F^{(l)}_\alpha}|u_q).
\]
If we consider $u_q(\sl_2)$ for example, and $\exp_+^\lambda$ corresponding to the positive simple root, we have $\exp_+^\lambda(E)=\exp^\lambda_+(K)=0$ and 
\[
\exp^\lambda_+(F)=F+\lambda\left(\frac{qK+q^{-1}K^{-1}}{q-q^{-1}}\right)E^{(l-1)}.
\]

As $E^{(l)}_\alpha$ fails to be primitive, the automorphism $\exp^\lambda_{\alpha}$ fails to be a Hopf map.  We have, in the ambient algebra $U_q$,
\[
\Delta(E^{(l)}_\alpha)=E^{(l)}_\alpha\ox 1+1\ox E^{(l)}_\alpha+\sum_{1\leq i\leq l-1}q^{-i(l-i)}K^iE^{(l-i)}\ox E^{(i)}
\]
and can define the element
\[
\O(E_\alpha)=\Delta(E^{(l)}_\alpha)-(E^{(l)}_\alpha\ox 1+1\ox E^{(l)}_\alpha)
\]
in $u_q\ox u_q$.  Note that $\O(E_\alpha)$ is square zero, and hence we can exponentiate any scaling $\lambda\O(E_\alpha)$ to arrive at a unit
\[
J^\lambda_\alpha=\exp(\lambda\O(E_\alpha))\in u_q\ox u_q.
\]
We define similarly $J^\lambda_{-\alpha}=\exp(\lambda\O(F_\alpha))$ for $\O(F_\alpha)=\Delta(F^{(l)}_\alpha)-(F^{(l)}_\alpha\ox 1-1\ox F^{(l)}_\alpha)$.  One can check easily from the expressions of $\O(E_\alpha)$ and $\O(F_\alpha)$ that
\[
(\epsilon\ox 1)(J^\lambda_{\pm\alpha})=(1\ox \epsilon)(J^\lambda_{\pm\alpha})=1.
\]

\begin{theorem}\label{thm:twstauts}
For an arbitrary simple root $\alpha$, and $\lambda\in \C$, the unit $J^\lambda_{\pm\alpha}$ is a twist for $u_q(\g)$.  Furthermore, each pair $(\exp^{\lambda}_{\pm\alpha},J_{\pm\alpha}^\lambda)$ is a twisted automorphism of $u_q(\g)$.
\end{theorem}

We will only prove the result for positive $\alpha$, the computation for $-\alpha$ being completely similar.  Let us first give a technical lemma.

\begin{lemma}\label{lem:1376}
The elements $(\Delta\ox 1)(\O(E_\alpha))$ and $\O(E_\alpha)\ox 1$ commute, as do the elements $(1\ox \Delta)(\O(E_\alpha))$ and $1\ox \O(E_\alpha)$.
\end{lemma}

\begin{proof}
Since $E_\alpha$ lay in the Hopf subalgebra $U_q(\sl_2)\subset U_q(\g)$ generated by $K_\alpha$ and the $E^{(n)}_\alpha$, $F^{(n)}_\alpha$, we may assume $\g=\sl_2$.  We may further restrict to the positive Borel $U_+$, in which $E^{(l)}$ is central.  Take $\O E=\O(E)$ and $pE^{(l)}=E^{(l)}\ox 1+1\ox E^{(l)}$.  We have now
\[
\begin{array}{l}
(\Delta\ox 1)(\O E)(\O E\ox 1)\\
=(\Delta\ox 1)(\O E)(\Delta E^{(l)}\ox 1)-(\Delta\ox 1)(\O E\ox 1)(pE^{(l)}\ox 1)\\
=(\Delta\ox 1)\left(\O E(E^{(l)}\ox 1)\right)-(\Delta\ox 1)(\O E\ox 1)(pE^{(l)}\ox 1)\\
=(\Delta\ox 1)\left((E^{(l)}\ox 1)\O E\right)-(pE^{(l)}\ox 1)(\Delta\ox 1)(\O E\ox 1)\\
=\left((\Delta\ox 1)(E^{(l)}\ox 1)-(pE^{(l)}\ox 1)\right)(\Delta\ox 1)(\O E)\\
=(\O E\ox 1)(\Delta\ox 1)(\O E).
\end{array}
\]
This gives the first proposed commutativity
\[
(\Delta\ox 1)(\O E)(\O E\ox 1)=(\O E\ox 1)(\Delta\ox 1)(\O E).
\]
The verification of the relation
\[
(1\ox\Delta)(\O E)(1\ox\O E)=(1\ox\O E)(1\ox\Delta)(\O E)
\]
is completely similar.
\end{proof}

Since all of the elements in the statement of Lemma~\ref{lem:1376} are nilpotent in $u_q^{\ox 3}$ we can now exponentiate to get
\begin{equation}\label{eq:1404}
\begin{array}{rl}
\exp\left((\Delta\ox 1)(\lambda\O E_\alpha)+(\lambda\O E_\alpha\ox 1)\right)&=\exp\left((\Delta\ox 1)(\lambda\O E_\alpha)\right)\exp\left(\lambda\O E_\alpha\ox 1\right)\\
&=(\Delta\ox 1)\left(\exp(\lambda\O E_\alpha)\right)\exp\left(\lambda\O E_\alpha\ox 1\right)\\
&=(\Delta\ox 1)\left(J_\alpha^\lambda\right)\left(J_\alpha^\lambda\ox 1\right)
\end{array}
\end{equation}
and
\begin{equation}\label{eq:1412}
\exp\left((1\ox\Delta)(\lambda\O E_\alpha)+(1\ox \lambda\O E_\alpha)\right) = (1\ox \Delta)(J_\alpha^\lambda)(1\ox J_\alpha^\lambda),
\end{equation}
for arbitrary $\lambda\in \C$.

\begin{proof}[Proof of Theorem~\ref{thm:twstauts}]
Again, we may assume $\g=\sl_2$.  By the above observations~(\ref{eq:1404}, \ref{eq:1412}) the dual cocycle condition for $J^\lambda_\alpha$ is equivalent to the equality
\[
(\Delta\ox 1)(\lambda\O E)+(\lambda\O E\ox 1)=(1\ox\Delta)(\lambda\O E)+(1\ox \lambda\O E).
\]
By dividing by $\lambda$ on both sides we may take $\lambda=1$.  We then see directly
\[
(\Delta\ox 1)(\O E)+(\O E\ox 1)=\sum_{\stackrel{0<i,j,k<l}{i+j+k=l}} q^{i(j+k)+jk}K^{i(j+k)}E^{(i)}\ox K^kE^{(j)}\ox E^{(k)}
\]
\[
=(1\ox\Delta)(\O E)+(1\ox \O E).
\]
Hence $J^\lambda_\alpha$ is a twist.
\par

As for compatibility with the automorphism $\exp^\lambda_\alpha$, we have the diagram
\[
\xymatrix{
u_q\ar[rr]^{\mathrm{ad}_{\lambda E^{(l)}_\alpha}}\ar[d]_\Delta & & u_q\ar[d]^\Delta\\
u_q\ox u_q\ar[rr]_{\mathrm{ad}_{\Delta \lambda E^{(l)}_\alpha}} & & u_q\ox u_q
}
\]
which implies the diagram
\[
\xymatrix{
u_q\ar[rr]^{\exp_\alpha^\lambda}\ar[d]_\Delta & & u_q\ar[d]^\Delta\\
u_q\ox u_q\ar[rr]_{\mathrm{exp}(\mathrm{ad}_{\Delta \lambda E^{(l)}_\alpha})} & & u_q\ox u_q.
}
\]
Since $\Delta E^{(l)}_\alpha=E^{(l)}_\alpha\ox 1+1\ox E^{(l)}_\alpha+\O(E_\alpha)$, we have 
\[
\mathrm{exp}(\mathrm{ad}_{\Delta \lambda E^{(l)}_\alpha})=(\exp_\alpha^\lambda\ox \exp_\alpha^\lambda)\mathrm{Ad}_{J^\lambda_\alpha},
\]
where $\mathrm{Ad}_u(x)=uxu^{-1}$, and the above diagram gives on elements
\[
(\exp_\alpha^\lambda\ox \exp_\alpha^\lambda)\Delta^{J^{-\lambda}_\alpha}(x)=\Delta(\exp_\alpha^\lambda(x)).
\]
Replace $x$ with $\exp^{-\lambda}_\alpha(x)$, compose with $(\exp_\alpha^{-\lambda}\ox \exp_\alpha^{-\lambda})$, and swap $\lambda$ for $-\lambda$ to find that $\Delta^{J^{\lambda}_\alpha}(\exp_\alpha^\lambda(x))=(\exp_\alpha^\lambda\ox \exp_\alpha^\lambda)\Delta(x)$.  So we see $\exp_{\alpha}^\lambda:u_q\to u_q^{J^\lambda_\alpha}$ is a Hopf map.
\end{proof}

One can check easily 
\[
(\exp_{\pm\alpha}^{\lambda},J_{\pm\alpha}^{\lambda})\cdot(\exp_{\pm\alpha}^{\lambda'},J_{\pm\alpha}^{\lambda'})=(\exp_{\pm\alpha}^{\lambda'+\lambda},J_{\pm\alpha}^{\lambda'+\lambda}).
\]
It follows that the assignment $\lambda\mapsto (\exp_{\pm\alpha}^{-\lambda},J_{\pm\alpha}^{-\lambda})$ gives a $1$-parameter subgroup in the $2$-group of twisted automorphisms for $u_q$, and hence a $1$-parameter subgroup $\C\to \underline{\mathrm{Aut}}(\rep(u_q))$ into the $2$-group of autoequivalences $\underline{\mathrm{Aut}}(\rep(u_q))$.  The negation here appears for technical reasons, but intuitively corrects the fact that the multiplication of twisted automorphisms defined above appears to be backwards.  We denote this $1$-parameter subgroup $\omega_{\pm\alpha}$.
\par

\begin{remark}
The algebra automorphisms appearing in the $1$-parameter subgroups $\omega_{\pm\alpha}$ can be recovered alternatively from the quantum coadjoint action of De Concini and Kac, via the reduction $U^{\mathrm{DK}}_q\to u_q$ from the non-divided-powers quantum group~\cite[Prop. 3.5]{deconcinikac89}.  So we are saying above that the induced quantum coadjoint action on $u_q$ extends naturally to an action on the tensor category $\rep(u_q)$.
\end{remark}

\subsection{Identification with the Arkhipov-Gaitsgory action} Take ${\Theta}$ the connected, simply connected, semisimple algebraic group with Lie algebra $\g$.  As a set we identify ${\Theta}$ with its $\mathbb{C}$-points.  Taking the (finite) dual of the exact sequence of Hopf algebras $\C\to u_q(\g)\to U_q(\g)\to U(\g)\to \C$ produces an exact sequence
\[
\C\to \O({\Theta})\to \O_q({\Theta})\to u_q(\g)^\ast\to \C
\]
with $\O({\Theta})$ laying in the center of the quantum function algebra~\cite[Thm. 6.3, Lem. 6.1]{deconcinilyubashenko94}.  According now to~\cite[Thm. 2.8]{arkhipovgaitsgory03} and~\cite[Prop. 4.1]{agp14} we have a tensor equivalence between the de-equivariantization $\mathrm{corep}\left(\O_q({\Theta})\right)_{\Theta}$ and $\rep(u_q(\g))$.  We take $\O=\O(\Theta)$ and $\O_q=\O_q(\Theta)$.
\par

Recall that the de-equivariantization is the category of finitely generated left $\O$-modules with a compatible right $\O_q$-coaction~\cite[Def. 3.7]{agp14}.  This category is monoidal under the product $\ox_{\O}$.  The action of ${\Theta}$ on itself by left translation, and pushing forward by the corresponding automorphisms of $\O$, gives an action of ${\Theta}$ on the de-equivariantization by tensor functors.  Rather, we have a canonical monoidal functor from ${\Theta}$ to the $2$-group of tensor autoequivalences of $\mathrm{corep}\left(\O_q\right)_{\Theta}$.  The equivalence $\mathrm{corep}\left(\O_q\right)_{\Theta}\overset{\sim}\to \rep(u_q)$ of~\cite{arkhipovgaitsgory03} is given by taking the fiber at the identity $?|_\epsilon=\C\ox_{\O}?$, and via this equivalence we get an action of ${\Theta}$ on $\rep(u_q)$.
\par

We let $\gamma_{\pm\alpha}$ denote the $1$-parameter subgroup in ${\Theta}$ given by exponentiating the root space $\g_{\pm\alpha}$.

\begin{proposition}\label{prop:G}
For any simple root $\alpha\in \Gamma$, the composite 
\[
\C\overset{\gamma_{\pm\alpha}}\longrightarrow {\Theta}\to \underline{\Aut}\left(\mathrm{corep}\left(\O_q\right)_{\Theta}\right)\overset{\mathrm{Ad}_{?|_\epsilon^{-1}}}\to\underline{\mathrm{Aut}}(\rep(u_q))
\]
is isomorphic to the $1$-parameter subgroup $\omega_{\pm\alpha}:\lambda\mapsto (\exp_{\pm\alpha}^{-\lambda},J_{\pm\alpha}^{-\lambda})$.
\end{proposition}

What one should mean by a general isomorphism of $1$-parameter subgroups is not exactly clear.  From our perspective we would like a family of natural isomorphisms between the two functors $\C\to \underline{\mathrm{Aut}}(\rep(u_q))$ which satisfy all obvious commutativity and additivity relations.  We will focus here only on the production of a natural family of natural isomorphisms which vary with $\lambda$.

\begin{proof}
We consider only the positive root $\alpha$.  Since high powers of $E^{(l)}_\alpha$ annihilate any finite dimensional representation, each function $f$ in the finite dual $\O_q$ will vanish on high powers of $E_\alpha^{(l)}$~\cite[Prop. 5.1]{lusztig89}.  Hence the exponent
\[
\exp(\lambda E_\alpha^{(l)}):\O_q\to \mathbb{C}
\]
is a well-defined function.  Restricting along the inclusion $\O\to \O_q$ recovers the point $\gamma_{\alpha}(\lambda)=\exp(\lambda e_\alpha)$ in ${\Theta}$.  Let us fix $x=x^\lambda=\gamma_{\alpha}(\lambda)$ and $v=v^\lambda=\exp(-\lambda E_\alpha^{(l)})$.
\par

Take $\mathrm{Ad}_{v}:\O_q\to \O_q$ the linear automorphism $f\mapsto v(f_1)f_2 v^{-1}(f_3)$, where $v^{-1}=\exp(\lambda E_\alpha^{(l)})$.  We note that $\mathrm{Ad}_v$ is a Hopf isomorphism from the cocycle twist of $\O_q$ via the $2$-cocycle $J_\alpha^{-\lambda}:\O_q\ox \O_q\to \C$ to $\O_q$, and so the sequence
\[
\mathrm{corep}(\O_q)\overset{J_\alpha^{-\lambda}}\to \mathrm{corep}((\O_q)_{J_\alpha^{-\lambda}})\overset{\mathrm{res}_{\mathrm{Ad}_v}}\to \mathrm{corep}(\O_q)
\]
is an equivalence.  This equivalence induces an equivalence on the equivariantization, where we additionally restrict the action of $\O$ along $\mathrm{Ad}_x$.  We denote this autoequivalence by $F^\lambda:\mathrm{corep}(\O_q)_{\Theta}\to \mathrm{corep}(\O_q)_{\Theta}$.
\par

We have now an isomorphism of monoidal functors $x^\lambda_\ast\overset{\cong}\to F^\lambda$ given on objects as the composite
\[
x_\ast^\lambda V\overset{\mathrm{comult}}\longrightarrow (x_\ast^\lambda V)\ox \O_q\overset{1\ox v^\lambda}\longrightarrow F^\lambda V.
\]
This equivalence is simply given by multiplying by the function $v^\lambda$, and we denote the isomorphism simply by $v^\lambda$.\footnote{The interested reader can check that the family of isomorphisms $\{v^\lambda\}_\lambda$ satisfies all desired commutativity and additivity relations to give an isomorphism between these two $1$-parameter subgroups in $\underline{\Aut}\left(\mathrm{corep}\left(\O_q\right)_{\Theta}\right)$.}  The quasi-inverse to the reduction $?|_\epsilon:\mathrm{corep}\left(\O_q\right)_{\Theta}\to\rep(u_q)$ is the induction-like functor $\mathrm{Ind}=(\O_q\ox ?)^{u_q}$, where $u_q$ acts on a product $\O_q\ox V$ diagonally by $h\cdot (f\ox v)=(fS(h_1)\ox h_2v)$.  We compose with the $2$-group map to $\underline{\mathrm{Aut}}(\rep(u_q))$ to get an induced isomorphism of $1$-parameter subgroups
\[
\dot{v}^\lambda=(?|_\epsilon)\circ v^\lambda\circ \mathrm{Ind}:(?|_\epsilon)\circ x_\ast^\lambda\circ \mathrm{Ind}\to (?|_\epsilon)\circ F^\lambda\circ \mathrm{Ind}.
\]
\par

Since $\mathrm{Ad}_x:\O\to \O$ preserves the counit, and $\mathrm{Ad}_v$ induces the automorphism $\exp^{-\lambda}_\alpha$ on the quotient $u_q^\ast$ (or rather its dual), taking the fiber at the identity gives 
\[
(F^\lambda\mathrm{Ind}(V))|_\epsilon=({_{\exp^{-\lambda}_\alpha}u_q^\ast}\ox V)^{u_q}={_{\exp^{-\lambda}_\alpha}(u_q^\ast\ox V)^{u_q}}
\]
for each $u_q$-representation $V$.  The subscript of $\exp^{-\lambda}_\alpha$ here means that we are restricting the action of $u_q$ along this automorphism.  But now the natural isomorphism of $u_q$-modules $ev_1\ox 1:(u_q^\ast\ox V)^{u_q}\to V$ given by the counit of $u_q^\ast$ produces the desired family of natural isomorphism
\[
\ddot{v}^\lambda:(?|_\epsilon)\circ x_\ast^\lambda\circ \mathrm{Ind}\overset{\dot{v}^\lambda}\longrightarrow (?|_\epsilon)\circ F^\lambda\circ \mathrm{Ind}\overset{ev_1\ox 1}\longrightarrow (\exp^{-\lambda}_\alpha,J^\lambda_\alpha).
\]
\end{proof}

Recall that $\Theta$ is generated by the $1$-parameter subgroups $\gamma_{\pm\alpha}$~\cite[Thm. 27.5]{humphreys12}.  Hence, after taking isoclasses, the group map ${\Theta}\to \mathrm{Aut}(\rep(u_q))$ is determined completely by its value on these $1$-parameter subgroups.  One can also show that the action of $G$ on $\rep(u_q)$ is determined up to unique isomorphism by these $1$-parameter subgroups, but this more limited information is enough for us to formulate Question~\ref{quest:G} below, which proposes a set of generators for the groupoid of twists of $u_q(\g)$.

\subsection{Autoequivalences of $\rep(u_q(\g))$ and the classification of twists}
\label{sect:G2}

Question~\ref{quest:G} below refers to gauge equivalence of twists.  We say two twists $J$ and $J'$ are gauge equivalent if there is a unit $v$ in $H$ with $J'=\Delta(v)J(v^{-1}\ox v^{-1})$.  We let $\mathsf{TW}(H)$ denote the groupoid of twists of $H$, with morphisms given by gauge equivalences.
\par

We note that the information of an isomorphism from the autoequivalence specified by a twisted automorphisms $(\phi,J)$ to that of $(\phi',J')$ is exactly the data of a unit $v\in H$ so that $J'=\Delta(v)J(v^{-1}\ox v^{-1})$ and $\phi'=\mathrm{Ad}_v\phi$.  We call such a unit a gauge equivalence of twisted automorphisms.  Hence we have naturally a $2$-group of twisted automorphisms with gauge equivalences.
\par

As is noted in~\cite{davydov10}, the groupoid $\mathsf{TW}(H)$ admits a well-defined right action of the $2$-group of twisted automorphisms.  This action is defined simply by $J'\cdot (\phi,J)=J\phi^{\ox 2}(J')$.  Take
\[
\tilde{{\Theta}}=\left\{\begin{array}{c}
\text{The $2$-subgroup of all twisted automorphisms which are iso-}\\
\text{morphic to an element in the image of }{\Theta}\text{, in }\underline{\mathrm{Aut}}(\rep(u_q))\end{array}\right\}
\]
\[
\hspace{5mm}=\left\{\begin{array}{c}
\text{The $2$-subgroup of twisted automorphisms which are}\\
\text{gauge equivalent to a product of the }(\exp_{\pm\alpha}^\lambda,J^\lambda_{\pm\alpha})\end{array}\right\}.
\]
Take also $\mathsf{BD}(u_q)\subset \mathsf{TW}(u_q)$ the full subcategory of Belavin-Drinfeld twists $\{J_{T,\S}\}_{T,\S}$.  The following question is also raised in~\cite{DEN}, where the authors investigate the algebraic structure of $\Aut(\rep(u_q))$, and autoequivalence groups of finite tensor categories in general.

\begin{question}\label{quest:G}
Is the groupoid of twists of the small quantum group generated by $\mathrm{BD}(u_q)$ and the $1$-paramater subgroups $\{(\exp_{\pm\alpha}^\lambda,J^{\lambda}_{\pm\alpha})\}_\lambda$?  Equivalently, is the inclusion $\mathsf{BD}(u_q)\cdot\tilde{{\Theta}}\to \mathsf{TW}(u_q)$ an equivalence?
\end{question}

\bibliographystyle{abbrv}

\end{document}